\documentclass[11pt]{amsart}
\usepackage{amsfonts}
\usepackage{amsmath}
\usepackage{setspace}
\usepackage{amsthm}
\usepackage{amssymb}
\usepackage[all]{xy}
\usepackage{color}
\usepackage{soul}
\usepackage{mathrsfs}
\usepackage{colortbl}
\usepackage{tikz}
\usetikzlibrary{arrows.meta, decorations.pathmorphing, backgrounds, positioning, fit, petri, patterns}
\usepackage[margin=1in]{geometry}

\newcommand{\Z}{{\mathbb{Z}}}
\newcommand{\C}{{\mathbb{C}}}
\newcommand{\R}{{\mathbb{R}}}

\newcommand{\e}{\varepsilon}
\newcommand{\pushout}{\ar@{}[ul(0.35)]|-{\ulcorner}}
\newcommand{\pullback}{\ar@{}[dr(0.35)]|-{\lrcorner}}

\DeclareMathOperator{\rep}{rep}
\newcommand{\AR}{A_{\R}}
\newcommand{\ARS}{A_{\R{,}S}}
\newcommand{\kk}{\ensuremath{\Bbbk}}
\newcommand{\repAR}{\rep_{\kk}(\AR)}
\newcommand{\repARS}{\rep_{\kk}(\ARS)}

\newcommand{\DbARS}{{\mathcal{D}^b(\ARS)}}
\newcommand{\CAR}{{\mathcal{C}(\AR)}}
\newcommand{\CARS}{{\mathcal{C}(\ARS)}}
\newcommand{\CsAR}{\mathcal C_{\overline{\R}}}
\newcommand{\CsARS}{\mathcal C_{\overline{\R,S}}}

\newcommand{\Cinftyclosed}{\mathcal C(A_{\overline{\infty}})}
\newcommand{\Ninftyclosed}{\mathbf N_{\overline{\infty}}}

\newcommand{\NRclosed}{\mathbf N_{\overline{\R}}}
\newcommand{\NRSclosed}{\mathbf N_{\overline{\R,S}}}

\newcommand{\KsplitCARS}{{K_0^{\text{split}}(\CARS)}}
\newcommand{\pmu}{p_{\overline{\mu}}}

\DeclareMathOperator{\Hom}{Hom}

\definecolor{linkylink}{RGB}{200,0,0}
\usepackage[colorlinks,citecolor=blue,linkcolor=linkylink]{hyperref}

\definecolor{gordanagreen}{RGB}{0,155,0}

\definecolor{readableyellow}{RGB}{200,200,0}

\newcommand{\id}{\mathop{\text{id}}}

\newcommand{\Ext}{\mathop{\text{Ext}}}

\newcommand{\pmset}{\{{-},{+}\}}
\newcommand{\Ind}{\text{Ind}}
\newcommand{\Edown}{\mathcal E^{\downarrow}}
\newcommand{\Eup}{\mathcal E^{\uparrow}}
\newtheorem{lemma}{Lemma}[subsection]
\newtheorem{proposition}[lemma]{Proposition}
\newtheorem{theorem}[lemma]{Theorem}
\newtheorem{cor}[lemma]{Corollary}

\newtheorem{thm}{Theorem}

\numberwithin{table}{subsection}

\theoremstyle{definition}
\newtheorem{definition}[lemma]{Definition}
\newtheorem{remark}[lemma]{Remark}
\newtheorem{example}[lemma]{Example}
\newtheorem{notation}[lemma]{Notation}

\newtheorem{construction}[lemma]{Construction}
\newtheorem{rulez}[lemma]{Rule}

\title[Continuous Quivers of Type $A$ (IV)]{Continuous Quivers of Type $A$ (IV)\\\ Continuous Mutation and Geometric Models of $\mathbf E$-clusters}
\author{Job Daisie Rock}
\address{Ghent, Belgium}
\email{\texttt{jobdrock@gmail.com}}
\date{\today}

\begin{document}

\begin{abstract}
This if the final paper in the series \emph{Continuous Quivers of Type $A$}.
In this part, we generalize existing geometric models of type $A$ cluster structures to the new $\mathbf E$-clusters introduced in part (III).
We also introduce an isomorphism of cluster theories and a weak equivalence of cluster theories.
Examples of both are given.
We use thes geometric models and isomorphisms of cluster theories to begin classifying continuous type $A$ cluster structures.
We also introduce a continuous generalization of mutation.
This encompasses mutation and (infinite) sequences of mutation.
Then we link continuous mutation to our earlier geometric models.
Finally, we introduce the space of mutations which generalizes the exchange graph of a cluster structure.
\end{abstract}

\maketitle

\tableofcontents

\section*{Introduction}
\subsection*{History} Cluster algebras were first introduced by Fomin and Zelevinksy in \cite{FZ1}.
In particle physics they can be used to study scattering diagrams (see work of Golden, Goncharov, Spradlin, Vergud, and Volovicha in \cite{GGSVV}).
The structure of cluster algebras was first categoricalized independently by two teams in 2006: Buan, Marsh, Reineke, Reiten, and Todorov in \cite{BMRRT}  and Caldero, Chapaton, and Schiffler in \cite{CCS}.
The first team's construction provided a way to construct a cluster category from the category of finitely generated representations of a Dynkin quiver and the second team's construction related the category to a geometric model.
This geometric model on a polygon was extended to the infinity-gon by Holm and J{\o}rgensen and the completed infinity-gon by Baur and Graz in \cite{HolmJorgensen} and \cite{BaurGraz}, respectively.
In \cite{FST}, Fomin, Shapiro, and Thurston expanded on \cite{CCS} and studies the relationship between triangulated surfaces and cluster algebras.
We refer the reader to \cite[Chapter 4.1]{Am} for the state of the art at the time of writing.
A continuous construction, both categorically and geometrically, was introduced by Igusa and Todorov in \cite{IgusaTodorov1}.
Structures relating to clusters are still activity studied (\cite{example1,example2,example3,example4}).
In particular, continuous structures were studied by Arkani-Hamed, He, Salvatori, and Thomas in \cite{AHST} and by Kulkarni, Matherne, Mousavand and the author in \cite{KMMR}.

In Part (I) of this series Igusa, Todorov, and the author introduced continuous quivers of type $A$, denoted $\ARS$, which generalize quivers of type $A$ \cite{IgusaRockTodorov1}.
Results about decomposition of pointwise finite-dimensional representations of such a quiver and the category of finitely-generated representations (denoted $\repARS$) were proven.
In Part (II) the author generalized the Auslander--Reiten quiver for finitely-generated representations of an $A_n$ quiver and its bounded derived category to the Auslander--Reiten space for $\repARS$ and its bounded derived category, denoted $\DbARS$ \cite{Rock}.
Results were proven about constructions of extensions in $\repARS$ and distinguished triangles in $\DbARS$ in relation to the Auslander--Reiten space.
In Part (III) Igusa, Todorov, and the author used Parts (I) and (II) to  classify which continuous quivers of type $A$ are derived equivalent, construct the new continuous cluster category (denoted $\CARS$) with $\mathbf E$-clusters, and generalize the notion of cluster structures to cluster theories \cite{IgusaRockTodorov2}.
It was shown that each element in an $\mathbf E$-cluster has none or one choice of mutation and the result of mutation yielded another $\mathbf E$-cluster.
It was also shown that some type $A$ cluster theories (recovered from existing cluster structures) can be embedded in this new cluster theory.

\subsection*{Contributions}
The final part of this series begins with a review of the relevant parts of the previous works.
Then, we define an isomorphism of cluster theories and a weak equivalence of cluster theories (Definition~\ref{def:equivalence of cluster theories}).
In Sections \ref{sec:straight AR} and \ref{sec:other orientations of AR}, we construct geometric models of $\mathbf E$-cluster theories from part (III) of this series \cite{IgusaRockTodorov2}.
We obtain an additive category $\CsARS$ and a pairwise compatibility condition $\NRSclosed$ on its indecomposables that induces the cluster theory $\mathscr{T}_{\NRSclosed}(\CsARS)$.
The purpose of the geometric models is to generalize triangulations of polygons and ideal triangulations of the hyperbolic plane, which encode several existing type $A$ cluster structures \cite{CCS,IgusaTodorov1}.
In particular, we want a connection to the cluster theory $\mathscr{T}_{\mathbf E}(\CARS)$ from \cite{IgusaRockTodorov2}.
We prove that the geometric models are ``correct'' in Theorem~\ref{thm:A} and then use them to prove Theorem~\ref{thm:B}.

\begin{thm}[Theorems \ref{thm:geometric model} and \ref{thm:general geometric model}]\label{thm:A}
Let $\ARS$ be a continuous quiver of type $A$.
The pairwise compatibility condition $\NRSclosed$ induces the $\NRSclosed$-cluster theory of $\CsARS$ and there is an isomorphism of cluster theories $(F,\eta):\mathscr T_{\NRSclosed}(\CsARS)\to \mathscr T_{\mathbf E}(\CARS)$.
\end{thm}

\begin{thm}[Corollary \ref{cor:sinks and sources count}]	\label{thm:B}
	Let $\ARS$ and $A_{\R,R}$ be continuous quivers of type $A$ such that one of the following is true: (i) $|S|=|R|$ and $|S|<\infty$, (ii) $S$ and $R$ are both bounded on exactly one side, or (iii) both $S$ and $R$ are indexed by $\Z$.
	Then $\mathscr{T}_{\mathbf{E}}(\CARS)\cong \mathscr{T}_{\mathbf E}(\mathcal C(A_{\R,R}))$.
\end{thm}

In Section \ref{sec:connections to E-mutations} we use the geometric models to show how one may visualize $\mathbf E$-mutations.
Some of these pictures are different from the usual ``swap diagonals on a quadrilateral'' that appears for triangulations of polygons and ideal triangulations of the hyperbolic plane.

In Section~\ref{sec:continuous mutation and mutation paths} we define a continuous generalization of mutation (Definition \ref{def:continuous mutation}) with two key motivations.
The first is to unify various ways of describing a sequence of mutations (possibly infinite as in \cite{BaurGraz}).
In Part (III), Igusa, Todorov, and the author show that the indecomposable objects that were projective in $\repAR$ form an $\mathbf E$-cluster but many of the elements are not $\mathbf E$-mutable \cite[Examples 4.3.2 and 4.4.1]{IgusaRockTodorov2}.
The second motivation for continuous mutation is to work around this obstruction so that we may mutate the cluster of projectives into the cluster of injectives as one usually does for type $A_n$.
In Section~\ref{sec:connections to E-mutation paths} we show how mutations and continuous mutations can be interpreted with these geometric models.

We use continuous mutation to define mutation paths (Definition \ref{def:mutation path}) and generalize the exchange graph of a cluster structure to the space of mutations for a cluster theory (Definition \ref{def:space of mutations}).
For a cluster theory $\mathscr{T}_{\mathbf P}(\mathcal C)$, we denote its space of mutations by $\mathbf P(\mathcal C)$.
\begin{thm}[Propositions \ref{prop:space of mutations},	\ref{prop:not Hausdorff}, and \ref{prop:path endpoints}]\label{thm:intro:space of mutations}
	Let $\mathscr{T}_{\mathbf P}(\mathcal C)$ be a cluster theory and $\mathbf P(\mathcal C)$ its the space of mutations.
	Then $\mathbf P(\mathcal C)$ is a non-Hausdroff topological space where each path begins and ends at a $\mathbf P$-cluster, up to homotopy.
\end{thm}

In Definition \ref{def:reachable} we define what it means for one cluster to be (strongly) reachable from another.
We then show we have achieved the goal of working around the afore-mentioned obstruction.

\begin{thm}[Theorem \ref{thm:projectives reach injectives}]\label{thm:intro:projectives reach injectives}\label{thm:D}
Consider the $\mathbf E$-cluster theory of $\CARS$ where $\ARS$ has the straight descending orientation.
The cluster of injectives is strongly reachable from the cluster of projectives.
\end{thm}

\subsection*{Future Work}
The exchange graph of an $A_n$ cluster structure is well-understood but the space of mutations for $\mathbf E$-clusters poses difficult question due to continuous mutations (Section~\ref{sec:space of mutations}).
However, preliminary calculations suggest the techniques to prove Theorem~\ref{thm:D} may be generalized to all continuous quivers $\ARS$ of type $A$ where $|S|<\infty$.

It is not yet clear which $\mathbf E$-cluster theories for continuous type $A$ quivers are equivalent.
Some theories are shown to be isomorphic (see Propositions \ref{prop:opposite cluster theories} and \ref{prop:reversing cluster theories}) but the whole classification is still open (Section \ref{sec:cluster theory classification}).

The next question to ask is, ``What about continuous types other than $A$?''
The next steps are continuous types $\widetilde{A}$ and $D$.
Each present their own complications to our constructions.
If one performs a similar constructions for continuous type $D$ then the resulting cluster theory should be similar to Igusa and Todorov's construction in \cite{IgusaTodorov0}.
Preliminary work by the C.~Paquette, E.~Y{\i}ld{\i}r{\i}m, and the author show that continuous representations of type $D$ decompose similarly to representations of a $D_n$ quver.
Also, Hanson and the author have proven that representations of $\widetilde{A}$ decompose analogously to representations of $\widetilde{A}_n$ \cite{HR20}.

\subsection*{Changes to this Version}
This version has been reorganized and party rewritten for readability.
The original version of this paper also contained a section that related several different cluster theories of type $A$, including the new $\mathbf E$-cluster theory.
Due to length and recent results, this section has been removed and then expanded into a separate work \cite{Rock2}.

\subsection*{Conventions}
Here we state conventions used throughout the paper.
We have a fixed field $\kk$ throughout.
When we say a ``Krull--Schmidt category'' we mean a ``skeletally small Krull--Schmidt additive category.''

Let $a<b\in\R\cup\{\pm \infty\}$.
By the notation $|a,b|$ we mean an interval subset of $\R$ whose endpoints are $a$ and $b$.
The `$|$'s indicate that the inclusion of $a$ or $b$ is not known or not relevant.

When we say ``arc'' in reference to a polygon with more than 3 sides, we mean what many call a ``diagonal.''
That is, a straight line from one vertex of the polygon to a nonadjacent vertex.
This terminology better coincides with definitions in Section \ref{sec:geometric models of E-clusters}.
We will still use ``diagonals'' for quadrilaterals when we talk about mutation.

\subsection*{Acknowledgements}
The majority of this work was completed while the author was a graduate student at Brandeis University and the author would like to thank the university for their hospitality.
The author would also like to thank Kiyoshi Igusa and Gordana Todorov for their guidance and support.
They would like to further thank Ralf Schiffler for organizing the Cluster Algebra Summer School at the University of Connecticut in 2017 where the idea for this series was conceived.
Finally, they would like to thank Eric Hanson for helpful discussions.

\section{Prerequisites from This Series}
In this section we revisit the most relevant definitions and theorems from parts (I) and (III) in this series, divided into two subsections.

\subsection{Continuous Quivers of Type $A$ and Their Representations}
In this section we state relevant definitions and theorems from part (I) of this series.
In particular, we provide a definition of a continuous quiver of type $A$, its representations, and its indecomposables.
The reader may use the picture in Figure \ref{fig:continuous quiver of type A} for intuition when reading the definition of a continuous quiver of type $A$.
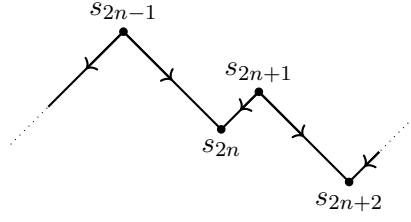
\begin{figure}
\centering	
\begin{tikzpicture}
\draw[dotted] (-.5,-.5) -- (.1,.1);
\draw[dotted] (4.3,-.7) -- (4.8,-.2);
\draw[thick] (0,0) -- (1,1) -- (2.3,-0.3) -- (2.8, 0.2) -- (4, -1) -- (4.4, -0.6);
\draw[thick,->] (1,1) -- (.5,.5);
\draw[thick,->] (1,1) -- (1.65,0.35);
\draw[thick,->] (2.8,0.2) -- (2.55,-0.05);
\draw[thick,->] (2.8,0.2) -- (3.4,-0.4);
\draw[thick,->] (4.4,-0.6) -- (4.2,-0.8);
\filldraw (1,1) circle[radius=.5mm];
\filldraw (2.3,-.3) circle[radius=.5mm];
\filldraw (2.8,.2) circle[radius=.5mm];
\filldraw (4,-1) circle[radius=.5mm];
\draw (1,1) node [anchor=south] {$s_{2n-1}$};
\draw (2.3,-0.3) node[anchor=north] {$s_{2n}$};
\draw (2.8,0.2) node[anchor=south] {$s_{2n+1}$};
\draw (4,-1) node[anchor=north] {$s_{2n+2}$};
\end{tikzpicture}
\caption{An example of a continuous quiver of type $A$.}\label{fig:continuous quiver of type A}
\end{figure}

\begin{definition}\label{def:AR}
A \ul{quiver of continuous type $A$}, denoted by $\ARS$, is a triple $(\R,S,\preceq)$, where:
\begin{enumerate}
\item 
\begin{enumerate}
\item$S\subset \R$ is a discrete subset, possibly empty, with no accumulation points.
\item Order on $S\cup\{\pm\infty\}$ is induced by the order of $\R$, and $-\infty<s<+\infty$ for $\forall s\in S$.
\item Elements of $S\cup\{\pm\infty\}$ are indexed by a subset of $\Z\cup\{\pm\infty\}$ so that $s_n$ denotes the element of 
$S\cup\{\pm\infty\}$ with index $n$. The indexing must adhere to the following two conditions:
\begin{itemize}
\item[i1] There exists $s_0\in S\cup\{\pm\infty\}$.
\item[i2] If $m\leq n\in\Z\cup\{\pm\infty\}$ and $s_m,s_n\in S\cup\{\pm\infty\}$ then for all $p\in\Z\cup\{\pm\infty\}$ such that $m\leq p \leq n$ the element $s_p$ is in $S\cup\{\pm\infty\}$.
\end{itemize}
\end{enumerate}
\item New partial order $\preceq$ on $\R$, which we call  the \ul{orientation} of $\AR$, is defined as:
\begin{itemize}
\item[p1\ ] The $\preceq$ order between consecutive elements of $S\cup\{\pm\infty\}$ does not change.
\item[p2\ ] Order reverses at each element of $S$.
\item[p3\ ] If $n$ is even $s_n$ is a sink.
\item[p3'] If $n$ is odd $s_n$ is a source.
\end{itemize}
\end{enumerate}
\end{definition}

\begin{definition}\label{def:representation}
Let $\ARS=(\R,S\preceq)$ be a continuous quiver of type $A$.
A \ul{representation} $V$ of $\ARS$ is the following data:
\begin{itemize}
\item A vector space $V(x)$ for {each} $x\in \R$.
\item For every pair $y\preceq x$ in $\AR$ a linear map $V(x,y):V(x)\to V(y)$ such that if $z\preceq y \preceq x$ then $V(x,z)=V(y,z)\circ V(x,y)$.
\end{itemize}
We say $V$ is \ul{pointwise finite-dimensional} if $\dim V(x) < \infty$ for all $x\in\R$.
\end{definition}

\begin{definition}\label{def: MI}
Let $\ARS$ be a continuous quiver of type $A$ and $I\subset \R$ be an interval.
We denote by $M_I$ the representation of $\ARS$ where
\begin{align*}
M_I(x) &= \left\{ \begin{array}{ll} \kk & x\in I \\ 0 & \text{otherwise} \end{array} \right. &
M_I(x,y) &= \left\{ \begin{array}{ll} 1_{\kk} & y\preceq x\in I \\ 0 & \text{otherwise.} \end{array} \right.
\end{align*}
We call $M_I$ an \ul{interval indecomposable}.
\end{definition}

We require the two following results from \cite{IgusaRockTodorov1} (the first recovers a result from \cite{BotnanCrawleyBoevey}).
\begin{theorem}[Theorems 2.3.2 and 2.4.13 in \cite{IgusaRockTodorov1}]\label{thm:GeneralizedBarCode}
Let $\ARS$ be a continuous quiver of type $A$.
For any interval $I\subset \R$, the representation $M_I$ of $\ARS$ is indecomposable.
Any indecomposable pointwise finite-dimensional representation of $\ARS$ is isomorphic to $M_I$ for some interval $I$.
Finally, any pointwise finite-dimensional representation $V$ of $\ARS$ is the direct sum of interval indecomposables.
\end{theorem}
\begin{theorem}[Theorem 2.1.6 and Remark 2.4.16 in \cite{IgusaRockTodorov1}]\label{thm:indecomp projs}
Let $P$ be a projective indecomposable in the category of pointwise finite-dimensional representations of a continuous quiver $\ARS$.
Then there exists $a\in \R\cup\{\pm\infty\}$ such that $P$ is isomorphic to one of $P_a$, $P_{(a}$, or $P_{a)}$, given by:
\begin{align*}
P_a(x) &= \left\{\begin{array}{ll} \kk & x \preceq a \\ 0 & \text{otherwise} \end{array}\right. &
P_a(x,y) &= \left\{\begin{array}{ll} 1_{\kk} & y\preceq x \preceq a \\ 0 & \text{otherwise} \end{array}\right. \\
P_{(a}(x) &= \left\{\begin{array}{ll}  \kk & x\preceq a\text{ and } x>a\text{ in }\R  \\ 0 & \text{otherwise} \end{array}\right. &
P_{(a}(x,y) &= \left\{\begin{array}{ll} 1_{\kk} & y\preceq x \preceq a \text{ and } x,y>a \\ 0 & \text{otherwise} \end{array}\right. \\
P_{a)}(x) &= \left\{\begin{array}{ll}  \kk & x\preceq a\text{ and }x<a\text{ in }\R  \\ 0 & \text{otherwise} \end{array}\right. &
P_{a)}(x,y) &= \left\{\begin{array}{ll} 1_{\kk} & y\preceq x \preceq a \text{ and }x,y<a \\ 0 & \text{otherwise} \end{array}\right.
\end{align*}
\end{theorem}

These allow us to define the category of finitely-generated representations:
\begin{definition}\label{def:fg reps}
Let $\ARS$ be a continuous quiver of type $A$.
By $\repARS$ we denote the full subcategory of pointwise finite-dimensional representations whose objects are finitely generated by the indecomposable projectives in Theorem \ref{thm:indecomp projs}.
\end{definition}

By \cite[Theorem 3.0.1]{IgusaRockTodorov1}, $\repARS$ is Krull--Schmidt with global dimension 1.

\subsection{Cluster Theories and Embeddings}
In this section we state the cluster theories content we need from Part (III) \cite{IgusaRockTodorov2}.
However, we first need just one result from Part (II).
\begin{proposition}[Proposition 5.1.2 in \cite{Rock}]\label{prop:derived is Krull--Schmidt}
Let $\ARS$ be a continuous quiver of type $A$.
Then $\DbARS$ is Krull--Schmidt.
The indecomposable objects are shifts of indecomposables in $\repARS$.
\end{proposition}

\begin{theorem}[Theorem A in \cite{IgusaRockTodorov2}]\label{thm:derived equivalence}
	$\ARS$ and $A_{\R,R}$ be continuous quivers of type $A$.
	Then $\DbARS$ is triangulated equivalent to $\mathcal D^b(A_{\R,R})$ if and only if (i) $S$ and $R$ are both finite, (ii) $S$ and $R$ both bounded on exactly one side, or (iii) $S$ and $R$ are both indexed by $\Z$.	
\end{theorem}

\begin{definition}\label{def:CAR}
The category $\CARS$ is the orbit category of the doubling of $\DbARS$ via almost-shift as in \cite[Definition 3.1.2]{IgusaRockTodorov2}.
\end{definition}
Importantly, the isomorphism classes of indecomposables in $\CARS$ are the same as if we had took the orbit of $\DbARS$ by shift.
That is, $V\cong V[1]$ for all indecomposables $V$ in $\CARS$.
The doubling process ensures $\CARS$ is a triangulated category.
Thus we have distinguished triangles in $\CARS$ of the form $Q_V\to P_V\to V\to Q_V$ where $Q_V\to P_V\to V\to 0$ is the minimal projective resolution of $V$ in $\repAR$.
Furthermore, for indecomposables $V$ and $W$ in $\CAR$, either $\Hom_{\CARS}(V,W)\cong \kk$ or $\Hom_{\CARS}(V,W)=0$ \cite[Proposition 3.1.2]{IgusaRockTodorov2}.
The authors of \cite{IgusaRockTodorov2} then defined $g$-vectors following J{\o}rgensen and Yakimov in \cite{JorgensenYakimov}.
\begin{definition}\label{def:g vector}
Let $V$ be an indecomposable in $\CAR$.
The \ul{g-vector} of $V$ is the element $[P_V]-[Q_V]$ in $\KsplitCARS$ where $Q_V\to P_V\to V\to 0$ is the minimal projective resolution of $V$ in $\repAR$.
\end{definition}
\begin{definition}\label{def:euler form}
For $[A]=\sum_i m_i[A_i]$ and $[B]=\sum_j n_j[B_j]$ in $\KsplitCARS$, define:
\begin{displaymath} \left\langle [A]\, ,\, [B] \right\rangle := \sum_i \sum_j \langle m_i[A_i]\, ,\, n_j[B_j] \rangle. \end{displaymath}
\end{definition}

The Euler form is used to define $\mathbf E$-compatibility and $\mathbf E$-clusters.
\begin{definition}\label{def:E-cluster}
{~}
\begin{itemize}
\item Let $V$ and $W$ be two indecomposables in $\CARS$ with g-vectors $[P_V]-[Q_V]$ and $[P_W]-[Q_W]$.
We say $\{V,W\}$ is \ul{$\mathbf E$-compatible} if
\begin{displaymath}
\langle [P_V]-[Q_V]\, ,\, [P_W]-[Q_W] \rangle \geq 0 \qquad \text{and} \qquad \langle [P_W]-[Q_W]\, ,\, [P_V]-[Q_V] \rangle \geq 0.
\end{displaymath}

\item A set $T$ is called \ul{$\mathbf E$-compatible} if for every $V,W\in T$ the set $\{V,W\}$ is $\mathbf E$-compatible.
If $T$ is maximally $\mathbf E$-compatible then we call $T$ an \ul{$\mathbf E$-cluster}.

\item Let $T$ be an $\mathbf E$-cluster and $V\in T$ such that there exists $W\notin T$ where $\{V,W\}$ is not $\mathbf E$-compatible but $(T\setminus\{V\})\cup\{W\}$ is $\mathbf E$-compatible.
Then we say $V$ is \ul{$\mathbf E$-mutable}.
The bijection (see Theorem \ref{thm:mutation theorem}) $T\to (T\setminus\{V\})\cup\{W\})$ given by $V\mapsto W$ and $X\mapsto X$ if $X\neq V$ is called an \ul{$\mathbf E$-mutation} or \ul{$\mathbf E$-mutation at $V$}.
\end{itemize}
\end{definition}

The following is used in Section \ref{sec:other orientations of AR}.
\begin{proposition}[Proposition 4.2.4 in \cite{IgusaRockTodorov2}]\label{prop:incompatible rectangles}
Let $V$ and $W$ be indecomposables in $\CARS$ and let $\widetilde{V}$ and $\widetilde{W}$ be the respective indecomposables in $\repARS$.
Then, up to symmetry, $V$ and $W$ are not $\mathbf E$-compatible if and only if there exists an extension $\widetilde{V}\hookrightarrow E \twoheadrightarrow \widetilde{W}$ such that $E\not\cong \widetilde{V}\oplus \widetilde{W}$.
\end{proposition}

The words $\mathbf E$-cluster and $\mathbf E$-mutation are justified with the following theorem.
\begin{theorem}[Theorem 4.3.8 in \cite{IgusaRockTodorov2}]\label{thm:mutation theorem}
Let $T$ be an $\mathbf E$-cluster and $V\in T$ $\mathbf E$-mutable with choice $W$.
Then $(T\setminus \{V\})\cup\{W\}$ is an $\mathbf E$-cluster and any other choice $W'$ for $V$ is isomorphic to $W$.
\end{theorem}

The key difference between $\mathbf E$-clusters and the usual cluster structures (such as those in \cite{BIRS}) is that not all $V$ in an $\mathbf E$-cluster $T$ need be mutable.
The authors only require there be none or one choice.
This is generalized to the abstract notion of cluster theories.

\begin{definition}[Definition 4.1.1 in \cite{IgusaRockTodorov2}]\label{def:cluster theory}
Let $\mathcal C$ be a Krull--Schmidt additive category and $\mathbf P$ a pairwise compatibility condition on its (isomorphism classes of) indecomposable objects.
Suppose that for each (isomorphism class of) indecomposable $X$ in a maximally $\mathbf P$-compatible set $T$ there exists none or one (isomorphism class of) indecomposable $Y$ such that $\{X,Y\}$ is not $\mathbf P$-compatible but $(T\setminus\{X\})\cup\{Y\}$ is $\mathbf P$-compatible. Then
\begin{itemize}
\item We call the maximally $\mathbf P$-compatible sets \ul{$\mathbf P$-clusters}.
\item We call a function of the form $\mu:T\to (T\setminus\{X\})\cup\{Y\}$ such that $\mu Z=Z$ when $Z\neq X$ and $\mu X=Y$ a \ul{$\mathbf P$-mutation} or \ul{$\mathbf P$-mutation at $X$}.
\item If there exists a $\mathbf P$-mutation $\mu:T\to (T\setminus\{X\})\cup\{Y\}$ we say $X\in T$ is \ul{$\mathbf P$-mutable}.
\item The subcategory $\mathscr T_{\mathbf P}(\mathcal C)$ of $\mathcal S\text{et}$ whose objects are $\mathbf P$-clusters and whose morphisms are generated by $\mathbf P$-mutations (and identity functions) is called \ul{the $\mathbf P$-cluster theory of $\mathcal C$}.
\item The functor $I_{\mathbf P,\mathcal C}:\mathscr T_{\mathbf P}(\mathcal C)\to \mathcal S\text{et}$ is the inclusion of the subcategory.
\end{itemize}
\end{definition}

\begin{remark}
We note three things immediately about Definition \ref{def:cluster theory}.
	\begin{itemize}
	\item The set $(T\setminus \{X\}) \cup \{Y\}$ must be maximally $\mathbf P$-compatible, so this does not need to be checked in practice.
	\item Since $\mathbf P$-clusters contain isomorphism classes of indecomposables as elements and $\mathcal C$ is skeletally small, the category $\mathscr{T}_{\mathbf P}(\mathcal C)$ is small.
	\item Finally, the pairwise compatibility condition $\mathbf P$ determines the cluster theory.
	\end{itemize}
\end{remark}

Thus we may say that $\mathbf P$ induces the cluster theory.
\begin{definition}[Definition 4.1.4 in \cite{IgusaRockTodorov2}]\label{def:tilting cluster theory}
Let $\mathcal C$ be a Krull--Schmidt category and $\mathbf P$ a pairwise compatibility condition such that $\mathbf P$ induces the $\mathbf P$-cluster theory of $\mathcal C$.
If, for every $\mathbf P$-cluster $T$ and $X\in T$, there is a $\mathbf P$-mutation at $X$ then we call $\mathscr T_{\mathbf P}(\mathcal C)$ the \ul{tilting $\mathbf P$-cluster theory}.
\end{definition}

Every cluster structure in the sense of \cite{BMRRT,BIRS} yields an tilting cluster theory.

\begin{definition}[Definition 4.1.9 in \cite{IgusaRockTodorov2}]\label{def:cluster embedding}
Let $\mathcal C$ and $\mathcal D$ be two Krull--Schmidt categories with respective pairwise compatibility conditions $\mathbf P$ and $\mathbf Q$.
Suppose these compatibility conditions induce the $\mathbf P$-cluster theory and $\mathbf Q$-cluster theory of $\mathcal C$ and $\mathcal D$, respectively.

Suppose there exists a functor $F:\mathscr T_{\mathbf P}(\mathcal C) \to \mathscr T_{\mathbf Q}(\mathcal D)$ such that $F$ is an injection on objects and an injection from $\mathbf P$-mutations to $\mathbf Q$-mutations.
Suppose also there is a natural transformation $\eta: I_{\mathbf P,\mathcal C} \to  I_{\mathbf Q, \mathcal D}\circ F$ whose morphisms $\eta_T: I_{\mathbf P,\mathcal C}(T) \to I_{\mathbf Q, \mathcal D}\circ F(T)$ are all injections.
Then we call $(F,\eta):\mathscr T_{\mathbf P}(\mathcal C) \to \mathscr T_{\mathbf Q}(\mathcal D)$ an \ul{embedding of cluster theories}.
\end{definition}

\section{Geometric Models of $\mathbf E$-clusters}\label{sec:geometric models of E-clusters}
In this section we construct geometric models of $\mathbf E$-clusters. 
In Section \ref{sec:straight AR} we address the straight descending orientation of $\AR$ and in Section \ref{sec:other orientations of AR} we address the rest of the orientations.
See \cite{Rock2} for a more general version of this technique.
We discuss the classification of cluster theories of continuous type $A$ in Section \ref{sec:cluster theory classification}.

Recall that an isomorphism of categories $F:\mathcal C\to \mathcal D$ has an inverse functor $G:\mathcal D\to \mathcal C$ such that $GF = 1_{\mathcal C}$ and $FG = 1_{\mathcal D}$; the compositions are \emph{equal} to the identity.

\begin{definition}\label{def:equivalence of cluster theories}
Let $\mathcal C$ and $\mathcal D$ be a Krull--Schmidt categories.
Let $\mathbf P$ and $\mathbf Q$ be pairwise compatibility conditions in $\mathcal C$ and $\mathcal D$ such that they, respectively, induce the cluster theories $\mathscr T_{\mathbf P}(\mathcal C)$ and $\mathscr T_{\mathbf Q}(\mathcal D)$.
A \ul{weak equivalence of cluster theories} is an embedding of cluster theories $(F,\eta):\mathscr T_{\mathbf P}(\mathcal C)\to\mathscr T_{\mathbf Q}(\mathcal D)$ such that $F$ is an isomorphism of categories.
We instead say $(F,\eta)$ is an \ul{isomorphism of cluster theories} if additionally each $\eta_T$ is an isomorphism.
\end{definition}

\begin{remark}\label{rmk:yes its strict}
An isomorphism of categories is ordinarily a very stringent requirement.
However, since every cluster theory is a groupoid the only real control we really have over comparing the ``size'' of each category is to insist they be identically the same via an isomorphism on objects.
And, since clusters in a cluster theory are sets of \emph{isomorphism classes of} objects in $\mathcal C$ and $\mathcal D$, respectively, we are already accounting for the type of equivalence with which we are familiar.
\end{remark}

We use the following lemma in Sections \ref{sec:straight AR} and \ref{sec:other orientations of AR}.
\begin{lemma}\label{lem:bijection and same compatibility gives equivalence}
Let $\mathcal C$ and $\mathcal D$ be Krull--Schmidt categories.
Let $\mathbf P$ be a pairwise compatibility condition in $\mathcal C$ such that $\mathbf P$ induces the cluster theory $\mathscr T_{\mathbf P}(\mathcal C)$ and let $\mathbf Q$ be a pairwise compatibility condition in $\mathcal D$.
Suppose
\begin{itemize}
\item there is a bijection $\Phi:\Ind(\mathcal C)\to \Ind(\mathcal D)$ and
\item for indecomposables $A$ and $B$ in $\mathcal C$, $\{A,B\}$ is $\mathbf P$-compatible if and only if $\{\Phi (A), \Phi(B)\}$ is $\mathbf Q$-compatible.
\end{itemize}
Then $\mathbf Q$ induces the cluster theory $\mathscr T_{\mathbf Q}(\mathcal D)$ and $\Phi$ induces an isomorphism of cluster theories $(F,\eta):\mathscr T_{\mathbf Q}(\mathcal D)\to \mathscr T_{\mathbf P}(\mathcal C)$.
\end{lemma}
\begin{proof}
Let $T$ be a maximally $\mathbf Q$-compatible set of $\mathcal D$-indecomposables and let $F(T) = \{\Phi^{-1}(A) \mid A\in T\}$.
First we show $F(T)$ is an $\mathbf P$-cluster.
Suppose $\{X\}\cup F(T)$ is $\mathbf P$-compatible.
Then $\{\Phi(X)\}\cup T$ is $\mathbf Q$-compatible.
However, $T$ is maximally $\mathbf Q$-compatible and so $\Phi(X)\in T$ and $X\in F(T)$.

Suppose there is $A\in T$ and $B\notin T$ such that $(T \setminus\{A\})\cup\{B\}$ is $\mathbf Q$-compatible.
Then $\{A,B\}$ is not $\mathbf Q$-compatible since $T$ is maximally $\mathbf Q$-compatible.
So $\{\Phi^{-1}(A),\Phi^{-1}(B)\}$ is not $\mathbf P$-compatible but $(F(T)\setminus\{\Phi^{-1}(A)\})\cup \{\Phi^{-1}(B)\}$ is $\mathbf P$-compatible.
This is a $\mathbf P$-mutation and so $(F(T)\setminus\{\Phi^{-1}(A)\})\cup \{\Phi^{-1}(B)\}$ is a $\mathbf P$-cluster.
Then by a similar argument to beginning of this proof, $(T\setminus\{A\})\cup \{B\}$ is maximally $\mathbf Q$-compatible.
Suppose there is $C\notin T$ such that $(T\setminus\{A\})\cup \{C\}$ is $\mathbf Q$-compatible.
Again, $\{A,C\}$ is not $\mathbf Q$-compatible and $(T\setminus\{A\})\cup \{C\}$ is maximally $\mathbf Q$-compatible.
However, this means $\Phi^{-1}(B)=\Phi^{-1}(C)$ and so $C=B$.
Therefore, $\mathbf Q$ induces the cluster theory $\mathscr T_{\mathbf Q}(\mathcal D)$.

We have already shown $F$ is a functor.
Suppose $T \neq T'$.
Then $T \cap T'\subsetneq T$ and $T \cap T'\subsetneq T'$
Using $\Phi^{-1}$ we see $F(T) \cap F(T')\subsetneq F(T)$ and $F(T) \cap F(T')\subsetneq F(T')$ which means $F(T)\neq F(T')$.
Suppose $L$ is a $\mathbf P$-cluster.
Then $\{\Phi(X) \mid X\in L\}$ is a $\mathbf Q$-cluster by a similar argument to that at the beginning of the proof.
Therefore, $F$ is an isomorphism of categories.
Finally, for each $\mathbf Q$-cluster $T$, we define $\eta_{T}:T\to F(T)$ by $A\mapsto \Phi^{-1}(A)$.
These are isomorphisms, as desired.
\end{proof}

\subsection{Straight orientaion: $\AR$}\label{sec:straight AR}
In this section we construct a geometric model of the cluster theory $\mathscr T_{\mathbf E}(\CAR)$ when $\AR$ has the straight descending orientation.
With this orientaiton there is a single frozen indecomposable in every $\mathbf E$-cluster (Definition \ref{def:E-cluster}): $P_{+\infty}$.
The geometric model of $\mathbf E$-clusters of this orientation is a generalization of the models in \cite{HolmJorgensen, BaurGraz}.
The generic arc in \cite{BaurGraz} is very similar to $P_{+\infty}$.

It is straightforward to check that for $M_{|a,b|}$ and $M_{|c,d|}$ where $a,b,c,d$ are all distinct the set $\{M_{|a,b|},M_{|c,d|}\}$ is not $\mathbf E$-compatible if and only if $a<c<b<d$ or $c<a<d<b$.
If $a<c<b<d$ we can draw the crossing arcs from $a$ to $b$ and from $c$ to $d$, for the ``macroscopic'' perspective, in Figure \ref{fig:macroscopic crossing}, both of which are always $\mathbf E$-compatible with $P_{+\infty}=M_{(-\infty,+\infty)}$, .
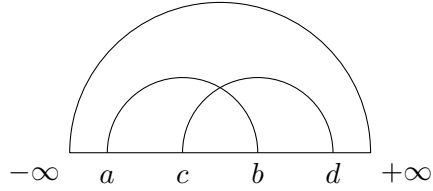
\begin{figure}
\centering
\begin{tikzpicture}[scale = 2]
\draw (0,0) -- (2,0) arc (0:180:1);
\draw (0.25,0) arc (180:0:.5);
\draw (1.75,0) arc (0:180:0.5);
\draw (0.25,-.05) node[anchor=north] {$a$};
\draw (1.25,0) node[anchor=north] {$b$};
\draw (0.75,-.05) node[anchor=north] {$c$};
\draw (1.75,0) node[anchor=north] {$d$};
\draw (0,0) node[anchor=north east] {$-\infty$};
\draw (2,0) node[anchor=north west] {$+\infty$};
\end{tikzpicture}
\caption{Schematic of crossing on the ``macroscopic'' scale.}\label{fig:macroscopic crossing}
\end{figure}

However, on the ``microscopic'' scale things are different.
Because we allow all types of intervals, we need two possible arc endpoints per $x\in\R$, but only one endpoint at each $-\infty$ and $+\infty$.
\begin{definition}\label{def:endpoints}
Let $\AR$ have the straight descending orientation.
In the set $\pmset$ we consider ${-} < {+}$ and denote an arbitrary element by $\e$, $\e'$, etc.
We give the set $\mathcal E:= (\R\times\pmset)\cup\{\pm\infty\}$ the total ordering where
\begin{itemize}
\item  $-\infty < (x,\pm) < +\infty$ for all $x\in \R$ and
\item $(x,\e) < (y,\e')$ if either $x<y$ or $x=y$ and $\e<\e'$.
\end{itemize}
For ease of notation we write $(-\infty,+)$ for $-\infty$ and $(+\infty,-)$ for $+\infty$.
We also write $\overline{a}$ to mean $(a,\e)$ for arbitrary $\e\in \{-,+\}$.
\end{definition}

\begin{definition}\label{def:straight arc}\label{def:arc indecomposable bijection}
Let $\mathcal A$ be the set $\{(\overline{a},\overline{b})\in \mathcal E\times \mathcal E \mid \overline{a}<\overline{b}\}$.
We call $\mathcal A$ the set of \ul{arcs}.
Let $M_{|a,b|}$ be the indecomposable in $\CAR$ that is the image of the indecomposable with the same name in $\repAR$.
We define $\Phi(M_{|a,b|})$ to be $(\overline{a},\overline{b})$ where
\begin{itemize}
\item $\overline{a}=(a,-)$ if $a\in |a,b|$ and $\overline{a}=(a,+)$ if $a\notin |a,b|$, and
\item $\overline{b}=(b,-)$ if $b\notin |a,b|$ and $\overline{b}=(b,+)$ if $b\in |a,b|$.
\end{itemize}
This defines $\Phi:\Ind(\mathcal C(\AR))\to \mathcal A$.
Note $\Phi(M_{[a,a]}) =((a,-),(a,+))$.
\end{definition}

We impose the following rule on our arcs.
\begin{rulez}\label{rule:easy cross}
Define a crossing function $\mathfrak c:\mathcal A\times\mathcal A\to \{0,1\}$.
\begin{equation}\label{eqn:easy cross}
	\mathfrak {c} (\alpha,\beta) = \begin{cases}
 		1 & \overline{a} <\overline{c}\leq \overline{b} < \overline{d} \text{ or } \overline{c} < \overline{a}\leq \overline{d} < \overline{b} \\
 		0 & \text{otherwise.}
	\end{cases}\tag{\ref{rule:easy cross}}
\end{equation}
If $\alpha\neq \beta$ and $\mathfrak c(\alpha,\beta)=1$ we say $\alpha$ and $\beta$ \ul{cross}.
Otherwise, we say $\alpha$ and $\beta$ \ul{do not cross}.
\end{rulez}

\noindent \textbf{!!} Notice the difference from the usual convention in the middle.
If two arcs meet from opposing sides we still consider them to cross.
This only happens on the ``microscopic'' scale.
I.e., for $a<b<d$, $(\overline{a},(b,-))$ and $((b,+),\overline{d})$ do not cross but any other combination of $+$ and $-$ for $\overline{b}$ cross (see Figure \ref{fig:microscopic crossing}).
\begin{figure}
\centering
\begin{tikzpicture}
\foreach \x in {0,...,3}
{
	\filldraw (3*\x,0) circle[radius=.4mm];
	\filldraw (3*\x + 0.5,0) circle[radius=.4mm];
}
\foreach \x in {0,3}
	\draw (\x,0) arc(0:70:1);

\foreach \x in {6,9}
	\draw (\x+0.5,0) arc(0:70:1);

\foreach \x in {3,9}
	\draw (\x,0) arc (180:110:1);

\foreach \x in {0,6}
	\draw (\x+0.5,0) arc (180:110:1);

\draw (0.25,1) node[anchor=south] {$-  +$};
\draw (3.25,1) node[anchor=south] {$-  -$};
\draw (6.25,1) node[anchor=south] {$+  +$};
\draw (9.25,1) node[anchor=south] {$+  -$};

\draw (0.25,0) node[anchor=north] {not};
\draw (0.25,-0.35) node[anchor=north] {crossing};

\foreach \x in {3,6,9}
	\draw (\x+0.25,-.15) node[anchor=north] {crossing};
\end{tikzpicture}
\caption{Possibilities for crossing and not crossing on the ``microscopic'' scale.}\label{fig:microscopic crossing}
\end{figure}
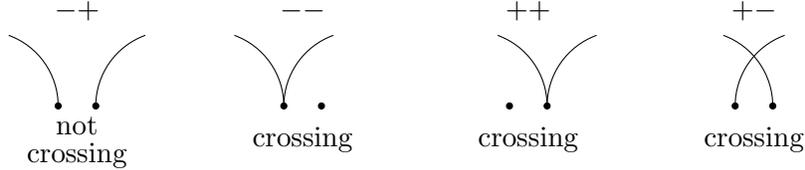

For the following proposition, recall the Definition~\ref{def:g vector}.

\begin{proposition}\label{prop:bijection of arcs and indecomposables}
The map $\Phi$ in Definition \ref{def:arc indecomposable bijection} is a bijection.
\end{proposition}
\begin{proof}
Suppose $M_{|a,b|}\not\cong M_{|c,d|}$.
Then $|a,b|\neq |c,d|$ and so one of the endpoints of the intervals must differ.
I.e., even if $a=c$ and $b=d$ then $a\notin |a,b|$ or $a\notin |c,d|$ or $b\notin |a,b|$ or $b\notin |c,d|$.
Then endpoints of the arcs associated to $M_{|a,b|}$ and $M_{|c,d|}$ are different.
Let $\alpha=(\overline{a},\overline{b})$ be an arc.
Then $\alpha=\Phi(M_{|a,b|})$ where $a\in |a,b|$ if and only if $\bar{a}=(a,-)$ and $b\in|a,b|$ if and only if $\overline{b}=(b,+)$.
Therefore $\Phi$ is both injective and surjective and so bijective.
\end{proof}

\begin{lemma}\label{lem:crossing = incompatible}
Let $M_{|a,b|}\neq M_{|c,d|}$ be indecomposables in $\CAR$, $\alpha=\Phi(M_{|a,b|})$, and $\beta=\Phi(M_{|c,d|})$.
Then $\{M_{|a,b|},M_{|c,d|}\}$ is $\mathbf E$-compatible if and only if $\mathfrak{c}(\alpha,\beta)=0$.
\end{lemma}
\begin{proof}
Suppose $\{M_{|a,b|}, M_{|c,d|}\}$ is not $\mathbf E$-compatible.
As we have discussed, if $a,b,c,d$ are all distinct then $a<c<b<d$ or $c<a<d<b$.
In either case it follows that $\alpha$ and $\beta$ cross.
Suppose $a=c$.
Since the $g$-vectors of $M_{|a,b|}$ and $M_{|c,d|}$ are not $\mathbf E$-compatible, (Definition~\ref{def:E-cluster}), we must have $a\notin |a,b|$ and $c\in |c,d|$ or vice versa.

Without loss of generality suppose $a\notin |a,b|$ and $c\in |c,d|$.
Then either $d < b$ or if $d=b$ then $d\notin |c,d|$ and $b\in |a,b|$.
In either case the arcs $\alpha$ and $\beta$ cross.
We can perform a similar argument starting with $b=d$ and see that $\alpha$ and $\beta$ cross.

Now suppose $\alpha$ and $\beta$ cross.
Then $\overline{a} <\overline{c}\leq \overline{b} < \overline{d}$ or $\overline{c} < \overline{a}\leq \overline{d} < \overline{b}$.
Without loss of generality assume the first.
Then if $a=c$, $a\in |a,b|$ but $c\notin |c,d|$.
Similarly if $b=d$ then $b\notin |a,b|$ and $d\in |c,d|$.
In all cases we see that the $g$-vectors of $M_{|a,b|}$ and $M_{|c,d|}$ are not $\mathbf E$-compatible and so the set $\{M_{|a,b|}, M_{|c,d|}\}$ is not $\mathbf E$-compatible.
\end{proof}

\begin{definition}\label{def:CR}
Let $\CsAR$ be the additive category whose indecomposable objects are $\mathcal A$.
We define $\Hom_{\CsAR}(\alpha,\beta)$ and composition $\alpha\stackrel{f}{\to}\beta\stackrel{g}{\to}\gamma$, for $f$ and $g$ nonzero, by
\begin{displaymath}
	\Hom_{\CsAR}(\alpha,\beta) = \begin{cases}
 		\kk & \mathfrak c(\alpha,\beta)=1 \\
 		0 & \text{otherwise}.	
 	\end{cases}
 	\qquad\quad
 	g\circ f = \begin{cases}
 		g\cdot f\in \kk & \alpha=\beta \text{ or } \beta=\gamma \\
 		0 & \text{otherwise}.
 	\end{cases}
\end{displaymath}
Extending bilinearly, we have a Krull--Schmidt category.
For $\alpha\neq\beta$, we define $\{\alpha,\beta\}$ to be $\mathbf N_\R$-compatible if and only if $\mathfrak c(\alpha,\beta)=0$.
\end{definition}



\begin{cor}[to Lemma \ref{lem:crossing = incompatible}]\label{cor:crossing = compatible}
Let $M_{|a,b|}$ and $M_{|c,d|}$ be in $\Ind(\CAR)$, $\alpha=\Phi(M_{|a,b|})$, and $\beta=\Phi(M_{|c,d|})$.
Then $\{\alpha,\beta\}$ is $\NRclosed$-compatible if and only if $\{M_{|a,b|},M_{|c,d|}\}$ is $\mathbf E$-compatible.
\end{cor}
\begin{proof}
The Corollary is immediately true if $\alpha=\beta$ or $M_{|a,b|}= M_{|c,d|}$.
Suppose $\alpha\neq\beta$.
Then $\{\alpha,\beta\}$ is $\NRclosed$-compatible if and only if $\mathfrak{c}(\alpha,\beta)=0$.
Now apply Lemma \ref{lem:crossing = incompatible}.
\end{proof}

\begin{theorem}\label{thm:geometric model}
The pairwise compatibility condition $\NRclosed$ induces the $\NRclosed$-cluster theory of $\CsAR$ and $\Phi$ induces the isomorphism of cluster theories $(F,\eta):\mathscr T_{\NRclosed}(\CsAR)\to \mathscr T_{\mathbf E}(\CAR)$.
\end{theorem}
\begin{proof}
We have shown there is a bijection $\Phi:\Ind(\CAR)~\to~\Ind(\CsAR)$ (Proposition \ref{prop:bijection of arcs and indecomposables}) and that $\{M_{|a,b|},M_{|c,d|}\}$ is $\mathbf E$-compatible if and only if $\{\Phi(M_{|a,b|}),  \Phi(M_{|c,d|})\}$ is $\NRclosed$-compatible (Corollary \ref{cor:crossing = compatible}).
By Lemma \ref{lem:bijection and same compatibility gives equivalence}, $\NRclosed$ induces the cluster theory $\mathscr T_{\NRclosed}(\CsAR)$ and we have the isomorphism of cluster theories given by $F(T) := \{\Phi^{-1}(\alpha) \mid \alpha\in T \}$ and $\eta_T(\alpha) := \Phi^{-1}(\alpha)$.
\end{proof}

If we remove the arc $((-\infty,+),(+\infty,-))$ from our geometric model, we still have a weak equivalence of cluster theories.

\subsection{Other orientations}\label{sec:other orientations of AR}
We now construct a geometric model of $\mathscr{T}_{\mathbf E}(\CARS)$ for orientations of $\ARS$ other than the straight orientation.
We take inspiration from the model of representations in \cite{BGMS}.
In the case of straight $\ARS$, we think of all the arcs as being directed: originating at the lower point and ending at the upper point.
We update our pictures from Rule \ref{rule:easy cross} to those in Figure \ref{fig:new microscopic crossing}.
When $\ARS$ has the straight descending orientation, these are the only possibilities.
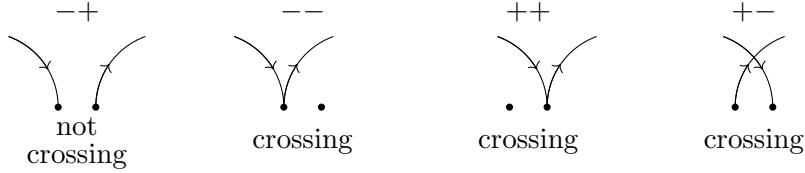
\begin{figure}
\centering
\begin{tikzpicture}
\foreach \x in {0,...,3}
{
	\filldraw (3*\x,0) circle[radius=.4mm];
	\filldraw (3*\x + 0.5,0) circle[radius=.4mm];
}
\foreach \x in {0,3}
{
	\draw (\x,0) arc(0:70:1);
	\draw[->] (\x-.658,.939) arc (70:30:1);
}

\foreach \x in {6,9}
{
	\draw (\x+0.5,0) arc(0:70:1);
	\draw[->] (\x-.158,.939) arc (70:30:1);
}

\foreach \x in {3,9}
{
	\draw (\x,0) arc (180:110:1);
	\draw[->] (\x,0) arc (180:145:1);
}

\foreach \x in {0,6}
{
	\draw (\x+0.5,0) arc (180:110:1);
	\draw[->] (\x+0.5,0) arc (180:145:1);
}

\draw (0.25,1) node[anchor=south] {$-  +$};
\draw (3.25,1) node[anchor=south] {$-  -$};
\draw (6.25,1) node[anchor=south] {$+  +$};
\draw (9.25,1) node[anchor=south] {$+  -$};

\draw (0.25,0) node[anchor=north] {not};
\draw (0.25,-0.35) node[anchor=north] {crossing};

\foreach \x in {3,6,9}
	\draw (\x+0.25,-.15) node[anchor=north] {crossing};
\end{tikzpicture}
\caption{Four of the possibilities for crossing and not crossing on the ``microscopic'' scale for all orientations of $\ARS$, using directed arcs.}\label{fig:new microscopic crossing}
\end{figure}

Now suppose $\ARS$ has some orientation other than straight descending or straight ascending.
We construct a set of endpoints $\mathcal E$ as the union of two sets: $\Edown$ and $\Eup$.
Recall in the definition of a continuous quiver of type $A$ (Definition \ref{def:AR}) that sinks have even index, $s_{2n}$, and sources have odd index, $s_{2n+1}$.
Recall also that if the sinks and sources of $\ARS$ are bounded below then $-\infty$ is assigned the next available index below and similarly for $+\infty$ when the sinks and sources are bounded above.
When the sinks and sources are not bounded below (above) we assign the index $-\infty$ to $-\infty$ ($+\infty$ to $+\infty$).

Recall $-<+$ in $\pmset$ and we write $\e$ to mean an arbitrary element in $\{\pmset\}$.
\begin{definition}\label{def:EAR}
The sets $\Edown$ and $\Eup$ are defined as follows, where each $s_m$ in the notation is a sink or source in $\ARS$ or one of $\pm \infty$ where appropriate:
\begin{align*}
\Edown := & \left( \{ x\in \R \mid \exists \text{ a sink and source } s_{2m} < x < s_{2m+1} \} \cup\{ (s_{2n-1},s_{2n}) \} \right) \times \pmset \\
\Eup := & \left( \{ x\in\R \mid \exists \text{ a source and sink } s_{2m-1} < x < s_{2m} \} \cup \{ (s_{2n}, s_{2n+1}) \} \right) \times \pmset
\end{align*}
We write $\bar{a}$ to mean an element $(a,\e)$ when $\e$ is unknown or understood from context. In this case $a$ may be in $\R$ or equal to some $(s_m,s_{m+1})$.
We define a total order on $\mathcal E:= \Edown\cup\Eup$:
\begin{itemize}
\item We say $(x,\e)<(y,\e')$ if $ x<y$ or $x=y$ and $\e<\e'$.
\item We say $((s_m,s_{m+1}),\e) < ((s_n,s_{n+1}),\e')$ if  $s_m<s_n$ or $s_m=s_n$ and $\e<\e'$.
\item We say $(x,\e) < ((s_m,s_{m+1}),\e')$ if $ x < s_m$ or if $s_m<x<s_{m+1}$ and $\e'=+$.
\item We say $((s_m,s_{m+1}),\e) < (y,\e')$ if $s_{m+1}<y$ or if $s_m<y<s_{m+1}$ and $\e=-$.
\end{itemize}
The set $\Edown$ has a maximal (respectively minimal) element if and only if $+\infty$ has an odd index (respectively $-\infty$ has an even index).
Dually, $\Eup$ has a maximal (respectively minimal) element if and only if $+\infty$ has an even index (respectively $-\infty$ has an odd index).
\end{definition}

\begin{definition}\label{def:arc}
Let $\mathcal A$ be the set $\{(\overline{a},\overline{b})\in\mathcal E\times\mathcal E \mid \overline{a}<\overline{b}\}$.
We call $\mathcal A$ the set of \ul{arcs}.
For an element $(\overline{a},\overline{b})\in\mathcal A$, we call $\overline{a}$ and $\overline{b}$ the \ul{endpoints}.
\end{definition}

\begin{example}\label{xmp:running xmp}
Let $\ARS$ have sinks $s_{-2}=-2, s_0=0,s_2=2$ and sources $s_{-1}=-1, s_1=1$.
Then $-\infty = s_{-3}$ and $+\infty = s_{3}$.
The set $\Edown$ has a maximum element and $\Eup$ has a minimum element.
In Figure \ref{fig:theta arc} we draw $\Edown$ and $\Eup$ using piece-wise linear curves in the plane and draw arcs on the ``macroscopic'' scale as lines between two points in $\mathcal E$.
For example, let $\alpha = (\overline{a},\overline{b})$ where $s_{-2}<a<s_{-1}$ and $s_1<b<s_2$.
Since $\overline{a}<\overline{b}$, we draw $\alpha$ oriented from $\overline{a}$ to $\overline{b}$.
\begin{figure}
\centering
\begin{tikzpicture}[scale = 1.5]
\draw (0,0) -- (1,1) -- (2.5,1) -- (3.5,0) -- (2.5,-1) -- (1,-1) -- (0,0);
\draw[->] (0,0) -- (0.5,0.5);
\draw[->] (1,-1) -- (0.5,-0.5);
\draw[->] (1,1) -- (1.75,1);
\draw[->] (2.5,-1) -- (1.75,-1);
\draw[->](2.5,1) -- (3,0.5);
\draw[->](3.5,0) -- (3,-0.5);
\draw (0,0) node[anchor=east] {$|s_{-3},s_{-2}|$};
\draw (1,1) node[anchor=south east] {$|s_{-2}, s_{-1}|$};
\draw (2.5,1) node[anchor=south west] {$|s_0,s_1|$};
\draw (1,-1) node[anchor=north east] {$|s_{-1},s_0|$};
\draw (2.5,-1) node[anchor=north west] {$|s_1,s_2|$};
\draw (3.5,0) node[anchor=west] {$|s_2,s_3|$};
\draw (.2,-.2) -- (3.3,0.2);
\draw[->](.2,-.2) -- (1.75,0);
\draw (1.75, 0) node[anchor=south] {$\alpha$};
\filldraw (.2,-.2) circle[radius=.4mm];
\filldraw (3.3,.2) circle[radius=.4mm];
\draw (.2,-.2) node[anchor=north east] {$\overline{a}$};
\draw (3.3,.2) node[anchor=south west] {$\overline{b}$};
\end{tikzpicture}
\caption{An example of an arc $\alpha$ from $\overline{a}$ to $\overline{b}$ for a particular $\ARS$ in Example \ref{xmp:running xmp}.}\label{fig:theta arc}
\end{figure}
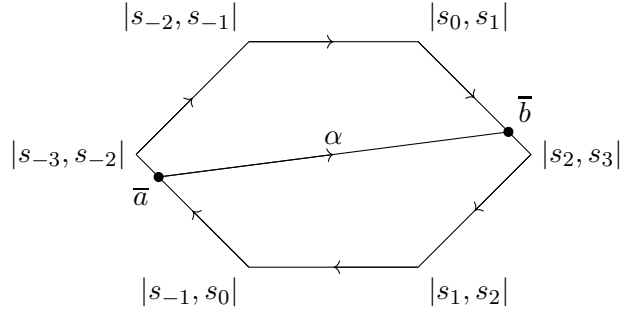
\end{example}

Recall that if the sinks and sources of $\ARS$ are unbounded below (respectively above) then no indecomposable in $\repARS$ may have $-\infty$ as a lower endpoint (respectively $+\infty$ as an upper endpoint) of its support.
Thus if we have $M_{|a,b|}$ and $a=-\infty$ (respectively $b=+\infty$) then we know the sinks and sources of $\ARS$ are bounded below (respectively above).
\begin{definition}\label{def:all bijection}
We now define $\Phi:\Ind(\CARS)\to\mathcal A$.
Let $M_{|a,b|}$ be an indecomposable in $\CARS$.
We define $\overline{x}$ and $\overline{y}$ in $\mathcal E$ beginning with $\overline{x}$.
\begin{itemize}
\item If $a\in\R$ is neither a sink nor a source then $\overline{x}=(a,\e)$ where $\e=-$ if and only if $a\in|a,b|$.
\item If $a=-\infty=s_m$ then $\overline{x}=((s_m,s_{m+1}),-)$.
\item If $-\infty<a=s_m$ and $a\in |a,b|$ then $\overline{x}=((s_m,s_{m+1}),-)$.
\item If $-\infty<a=s_m$ and $a\notin |a,b|$ then $\overline{x}=(|s_{m-1},s_m|,+)$.
\end{itemize}
Now, $\overline{y}$.
\begin{itemize}
\item If $b\in\R$ is neither a sink nor a source then $\overline{y}=(b,\e)$ where $\e=+$ if and only if $b\in|a,b|$.
\item If $b=+\infty=s_n$ then $\overline{y}=(|s_{n-1},s_n|,+)$.
\item If $+\infty > b=s_n$ and $b\in |a,b|$ then $\overline{y}=(|s_{n-1},s_n|,+)$.
\item If $+\infty > b=s_n$ and $b\notin |a,b|$ then $\overline{y}=((s_n,s_{n+1}),-)$.
\end{itemize}
\end{definition}

\begin{proposition}\label{prop:all bijection}
The function $\Phi$ in Definition \ref{def:all bijection} is a bijection.
\end{proposition}
\begin{proof}
Let $M_{|a,b|}\not\cong M_{|c,d|}$ be indecomposables in $\CARS$.
Let $(\overline{x},\overline{y})=\Phi(M_{|a,b|})$ and $(\overline{z},\overline{w})=\Phi(M_{|c,d|})$.
Using the definition it is straightforward to check that if $a\neq c$ or $b\neq d$ then $(\overline{x},\overline{y})\neq (\overline{z},\overline{w})$.
Now suppose $a=c$ and $b=d$.
Since $M_{|a,b|}\not\cong M_{|c,d|}$ the endpoints of $|a,b|$ and $|c,d|$ must differ by at least one point.
By symmetry and possibly reversing the roles of $M_{|a,b|}$ and $M_{|c,d|}$, assume $a\in|a,b|$ and $c\notin |a,b|$.
Then $\overline{x}\neq \overline{z}$ and so $(\overline{x},\overline{y})\neq (\overline{z},\overline{w})$.
Thus, $\Phi$ is injective.

Let $\alpha=(\overline{x},\overline{y})$ be an arc in $\mathcal A$.
We now construct an interval $|a,b|$ such that $\Phi(M_{|a,b|}) = \alpha$.
If $x\in \R$ and $x$ is neither a sink nor a source then we let $a=x$ and $a\in|a,b|$ if and only if $\e=-$.
If $y\in \R$ and $y$ is neither a sink nor a source then we let $b=y$ and $b\in|a,b|$ if and only if $\e'=+$.

Suppose $x=(s_m,s_{m+1})$.
If $\e=+$, then either $y\in\R$ is greater than $s_{m+1}$ or $b=(s_n,s_{n+1})$ where $n > m$.
In this case we let $a=s_{m+1}$ and $a\notin |a,b|$.
If $\e=-$, then either $y\in\R$ is greater than $s_m$ or $y=(s_n,s_{n+1})$ where $n\geq m$; if $n=m$ then $\overline{y}=((s_{m-1},s_m),+)$.
In this case if $s_m=-\infty$ then we let $a=-\infty$ and note $a\notin|a,b|$.
If $s_m>-\infty$ then we let $a=s_m$ and $a\in|a,b|$.
We perform the dual constructions for $\overline{y}$ and $b$ as well.

Now we have $a\leq b$ and the requirements for $|a,b|$ to contain either $a$ or $b$.
We need to check $a=b$ to ensure that in this case $a,b\in|a,b|$ by our construction.
If $a=b\in\R$ is neither a sink nor a source then $\alpha=\{(a,-),(a,+)\}$ and so $|a,b|=[a,a]$.
If $a=b\in\R$ is a sink or a source let $s_n=a=b$.
Then we have that $|a,b|=\{s_n\}$.
Thus $\Phi$ is surjective and so bijective.
\end{proof}

Our rules for crossing are more complicated than before.
The cases are: straightforward (Rule \ref{rule:straightforward cross}), the ``macroscopic'' (Rule \ref{rule:macro cross}), and the ``microscopic'' (Rule \ref{rule:micro cross}).

\begin{rulez}\label{rule:straightforward cross}
Let $\alpha$ and $\beta$ be arcs with endpoints in $\mathcal E$.
\begin{itemize}
\item If both $\alpha$ and $\beta$ have endpoints in $\Edown$ then we follow Rule \ref{rule:easy cross}.
\item  If both $\alpha$ and $\beta$ have endpoints in $\Eup$ then we follow Rule \ref{rule:easy cross}.
\item If $\alpha$ has endpoints in $\Edown$ and $\beta$ has endpoints in $\Eup$ then we say $\alpha$ and $\beta$ do not cross.
\end{itemize}
\end{rulez}

\begin{example}[Example of Rule \ref{rule:straightforward cross}]\label{xmp:straightforward cross}
Let $\ARS$ have sinks $s_{-2}=-2, s_0=0,s_2=2$ and sources $s_{-1}=-1, s_1=1$ with $-\infty = s_{-3}$ and $+\infty = s_{3}$ as in Example \ref{xmp:running xmp}.
Let $\alpha^{\downarrow}$ and $\beta^{\downarrow}$ be crossing arcs with endpoints in $\Edown$.
Let $\alpha^{\uparrow}$ and $\beta^{\uparrow}$ be crossing arcs with endpoints in $\Eup$.
\begin{figure}
\centering
\begin{tikzpicture}[scale = 1.5]
\draw (0,0) -- (1,1) -- (2.5,1) -- (3.5,0) -- (2.5,-1) -- (1,-1) -- (0,0);
\draw[->] (0,0) -- (0.5,0.5);
\draw[->] (1,-1) -- (0.5,-0.5);
\draw[->] (1,1) -- (1.75,1);
\draw[->] (2.5,-1) -- (1.75,-1);
\draw[->](2.5,1) -- (3,0.5);
\draw[->](3.5,0) -- (3,-0.5);
\draw (0,0) node[anchor=east] {$|s_{-3},s_{-2}|$};
\draw (1,1) node[anchor=south east] {$|s_{-2}, s_{-1}|$};
\draw (2.5,1) node[anchor=south west] {$|s_0,s_1|$};
\draw (1,-1) node[anchor=north east] {$|s_{-1},s_0|$};
\draw (2.5,-1) node[anchor=north west] {$|s_1,s_2|$};
\draw (3.5,0) node[anchor=west] {$|s_2,s_3|$};
\draw (0.35,0.35) arc (225:360:.92);
\draw[->] (0.35,0.35) arc (225:292.5:.92);
\draw (0.35,-0.35) arc (135:0:.92);
\draw[->] (0.35,-0.35) arc (135:67.5:0.92);
\draw (1.58,1) arc (180:315:.92);
\draw[->] (1.58,1) arc (180:247.5:0.92);
\draw (1.58,-1) arc (180:45:.92);
\draw[->] (1.58,-1) arc (180:112.5:0.92);
\draw (.9,0.1) node [anchor=south west] {$\alpha^{\uparrow}$};
\draw (2.6,.1) node [anchor=south east] {$\beta^{\uparrow}$};
\draw (.9,-.1) node [anchor=north west] {$\alpha^{\downarrow}$};
\draw (2.6,-.1) node [anchor=north east] {$\beta^{\downarrow}$};
\end{tikzpicture}
\caption{A visual depiction of why ``upper'' and ``lower'' arcs (in Example \ref{rule:straightforward cross}) do not cross. We see $\alpha^{\downarrow}$ and $\beta^{\downarrow}$ cross each other but cross neither $\alpha^{\uparrow}$ nor $\beta^{\uparrow}$.}\label{fig:upper and lower thetas}
\end{figure}
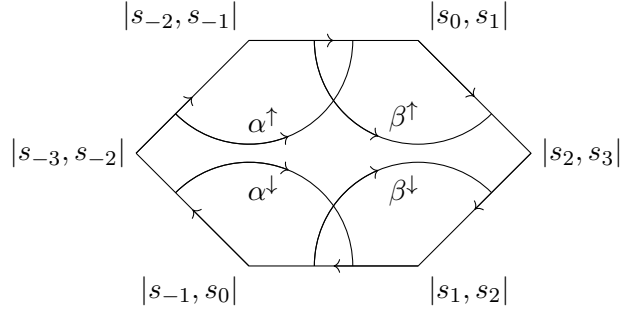
In Figure \ref{fig:upper and lower thetas} we see why the ``upper'' and ``lower'' arcs to not cross.
\end{example}

\begin{rulez}\label{rule:macro cross}
Let $\alpha$ and $\beta$ be arcs in $\mathcal A$.
Suppose $\alpha=(\overline{a},\overline{b})$ has one endpoint in $\Edown$ and the other in $\Eup$.
Let $\beta = (\overline{c},\overline{d})$.
We assume $\overline{a}$, $\overline{b}$, $\overline{c}$, and $\overline{d}$ are all distinct.
\begin{enumerate}
\item Suppose $\overline{c},\overline{d}\in\Edown$. If $\overline{c} < \overline{x} < \overline{d}$, where $\{\overline{x}\}=\{\overline{a},\overline{b}\}\cap\Edown$, we say $\alpha$ and $\beta$ cross.
\item Suppose $\overline{c},\overline{d}\in\Eup$. If $\overline{c} < \overline{x} < \overline{d}$, where $\{\overline{x}\}=\{\overline{a},\overline{b}\}\cap\Eup$, we say $\alpha$ and $\beta$ cross.
\item\label{itm:up and down} Suppose either (i) $\overline{a},\overline{c}\in\Edown$ and $\overline{b},\overline{d}\in\Eup$ or (ii) $\overline{a},\overline{c}\in\Eup$ and $\overline{b},\overline{d}\in\Edown$.
	If $\overline{a}<\overline{c}<\overline{d}<\overline{b}$ or $\overline{c}<\overline{a}<\overline{b}<\overline{d}$ we say $\alpha$ and $\beta$ cross.
\item\label{itm:up and down 2} Suppose either (i) $\overline{a},\overline{d}\in\Edown$ and $\overline{b},\overline{c}\in\Eup$ or (ii) $\overline{a},\overline{d}\in\Eup$ and $\overline{b},\overline{c}\in\Edown$.
	If $\overline{a}<\overline{d}$ and $\overline{c}<\overline{b}$ we say $\alpha$ and $\beta$ cross.
\end{enumerate}
\end{rulez}

Notice that $\overline{a}<\overline{d}$ and $\overline{c}<\overline{b}$ is not enough for Rule \ref{rule:macro cross}(\ref{itm:up and down}).
In this case, if $\overline{c}<\overline{a}<\overline{d}<\overline{b}$, for example, then $\alpha$ and $\beta$ \emph{do not} cross.
However, if $\alpha$ and $\beta$ cross, we must have $\overline{a}<\overline{d}$ and $\overline{c}<\overline{b}$.
See Example \ref{xmp:macro cross} and Figure \ref{fig:robust thetas}.

\begin{example}[Example of Rule \ref{rule:macro cross}]\label{xmp:macro cross}
Let $\ARS$ have sinks $s_{-2}=-2, s_0=0,s_2=2$ and sources $s_{-1}=-1, s_1=1$ with $-\infty = s_{-3}$ and $+\infty = s_{3}$ as in Example \ref{xmp:running xmp}.
For $\alpha$ in $\mathcal A$ such that one endpoint is in $\Edown$ and the other in $\Eup$, see Figure \ref{fig:robust thetas} for several examples of Rule \ref{rule:macro cross}.
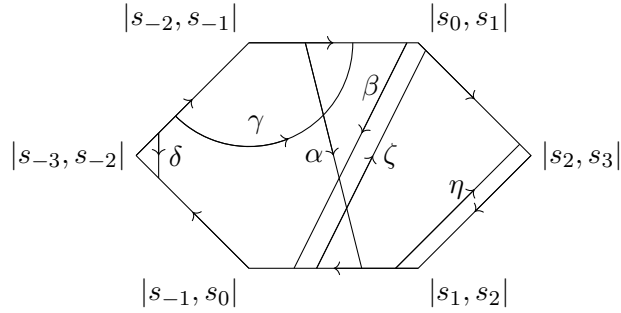
\begin{figure}
\centering
\begin{tikzpicture}[scale = 1.5]
\draw (0,0) -- (1,1) -- (2.5,1) -- (3.5,0) -- (2.5,-1) -- (1,-1) -- (0,0);
\draw[->] (0,0) -- (0.5,0.5);
\draw[->] (1,-1) -- (0.5,-0.5);
\draw[->] (1,1) -- (1.75,1);
\draw[->] (2.5,-1) -- (1.75,-1);
\draw[->](2.5,1) -- (3,0.5);
\draw[->](3.5,0) -- (3,-0.5);
\draw (0,0) node[anchor=east] {$|s_{-3},s_{-2}|$};
\draw (1,1) node[anchor=south east] {$|s_{-2}, s_{-1}|$};
\draw (2.5,1) node[anchor=south west] {$|s_0,s_1|$};
\draw (1,-1) node[anchor=north east] {$|s_{-1},s_0|$};
\draw (2.5,-1) node[anchor=north west] {$|s_1,s_2|$};
\draw (3.5,0) node[anchor=west] {$|s_2,s_3|$};
\draw (1.5,1) -- (2,-1); 
\draw[->] (1.5,1) -- (1.75,0);
\draw (2.4,1) -- (1.4,-1); 
\draw[->] (2.4,1) -- (2,.2);
\draw (0.35,0.35) arc (225:360:.92); 
\draw[->] (0.35,0.35) arc (225:292.5:.92);
\draw (.2,.2) -- (.2,-.2); 
\draw[->] (.2,.2) -- (.2,0);
\draw (1.6,-1) -- (2.565,0.93); 
\draw[->] (1.6,-1) -- (2.1, 0);
\draw (2.3, -1) -- (3.4,.1); 
\draw[->] (2.3, -1) -- (3, -.3);
\draw (1.75,0) node[anchor=east] {$\alpha$};
\draw (2.25,.6) node[anchor=east] {$\beta$};
\draw (.9,0.1) node [anchor=south west] {$\gamma$};
\draw (.2,0) node[anchor=west] {$\delta$};
\draw (2.1,0) node [anchor=west] {$\zeta$};
\draw (3, -.3) node[anchor=east] {$\eta$};
\end{tikzpicture}
\caption{Consider $\alpha$ from Example \ref{xmp:macro cross}. We have $\delta$ and $\eta$, examples of arcs with endpoints in both $\Edown$ and $\Eup$, that do not cross $\alpha$. We see $\delta$ goes ``top to bottom'' and $\eta$ goes ''bottom to top.'' We have $\beta$ and $\zeta$, also with mixed endpoints, that cross $\alpha$. We see $\beta$ goes ``top to bottom'' and $\zeta$ goes ``bottom to top.'' Finally, we have $\gamma$, with both endpoints in $\Eup$, that crosses $\alpha$ but not $\beta$. From Rule \ref{rule:macro cross}.}\label{fig:robust thetas}
\end{figure}
\end{example}

The only case not covered by Rules \ref{rule:straightforward cross} and \ref{rule:macro cross} is when two arcs share an endpoint.
\begin{rulez}\label{rule:micro cross}
Let $\alpha=(\overline{a},\overline{b})$ and $\beta=(\overline{c},\overline{d})$ be in $\mathcal A$.
We have four cases: $\overline{a} = \overline{c}$, $\overline{a}=\overline{d}$, $\overline{b}=\overline{c}$, and $\overline{b}=\overline{d}$.
(If two equalities hold at once we have $\alpha=\beta$.)
\begin{itemize}
\item If $\overline{a}=\overline{d}$ or $\overline{b}=\overline{c}$ then we say $\alpha$ and $\beta$ cross.
\item If $\overline{a}=\overline{c}$ or $\overline{b}=\overline{d}$ then we say $\alpha$ and $\beta$ do not cross.
\end{itemize}
See Figure \ref{fig:microscopic advanced} for a visual depiction of this rule.
\end{rulez}

\begin{definition}\label{def:zigzag crossing}
	Define the crossing function $\mathfrak c:\mathcal A\times\mathcal A\to\{0,1\}$ by
	\begin{displaymath}
		\mathfrak c(\alpha,\beta) = \begin{cases}
 		1 & (\alpha = \beta) \text{ or } (\alpha \text{ and } \beta \text{ cross according to Rules \ref{rule:easy cross}, \ref{rule:macro cross}, and \ref{rule:micro cross}})\\
 		0 & (\alpha\neq \beta) \text{ and } (\alpha\text{ and }\beta\text{ do not cross according to Rules \ref{rule:easy cross}, \ref{rule:macro cross}, and \ref{rule:micro cross}}).
 		\end{cases}
	\end{displaymath}
	For $\alpha\neq\beta$, if $\mathfrak{c}(\alpha,\beta)=1$ we say $\alpha$ and $\beta$ \ul{cross}.
	Otherwise, we say $\alpha$ and $\beta$ \ul{do not cross}.
\end{definition}

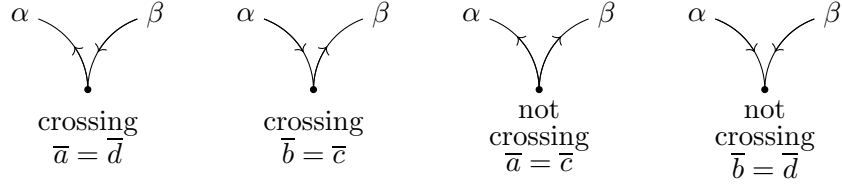
\begin{figure}
\centering
\begin{tikzpicture}
\foreach \x in {0,3,6,9}
{
	\filldraw (\x,0) circle[radius=.4mm];
	\draw (\x,0) arc(0:70:1);
	\draw (\x,0) arc (180:110:1);
	\draw (\x-0.9,1) node {$\alpha$};
	\draw (\x+0.9,1) node {$\beta$};
}

\draw[->] (3-.658,.939) arc (70:30:1);
\draw[->] (3,0) arc (180:145:1);
\draw[->] (0.658,.939) arc (110:150:1);
\draw[->] (0,0) arc (0:35:1);

\draw[->] (9-.658,.939) arc (70:30:1);
\draw[->] (6,0) arc (180:135:1);
\draw[->] (9.658,.939) arc (110:150:1);
\draw[->] (6,0) arc (0:45:1);

\foreach \x in {0,3}
	\draw (\x,-.15) node[anchor=north] {crossing};
	
\draw (0,-.5) node[anchor=north] {$\overline{a}=\overline{d}$};
\draw (3,-.5) node[anchor=north] {$\overline{b}=\overline{c}$};

\foreach \x in {6,9}
{
	\draw (\x,0) node[anchor=north] {not};
	\draw (\x,-.35) node[anchor=north] {crossing};
}

\draw (6,-.7) node[anchor=north] {$\overline{a}=\overline{c}$};
\draw (9,-.7) node[anchor=north] {$\overline{b}=\overline{d}$};
\end{tikzpicture}
\caption{A depiction of Rule \ref{rule:micro cross}.}\label{fig:microscopic advanced}
\end{figure}

We are now ready to prove the following lemma.

\begin{lemma}\label{lem:general crossing = incompatible}
Let $M_{|a,b|}\neq M_{|c,d|}$ be indecomposables in $\CARS$.
Then $\{M_{|a,b|},M_{|c,d|}\}$ is $\mathbf E$-compatible if and only if $\mathfrak{c}(\Phi(M_{|a,b|}),\Phi(M_{|c,d|}))=0$.
\end{lemma}
\begin{proof}
\ul{Setup.}
Let $\alpha=(\overline{x},\overline{y})=\Phi(M_{|a,b|})$ and $\beta=(\overline{z},\overline{w})=\Phi(M_{|c,d|})$ and note that by Lemma~\ref{prop:all bijection}, $\alpha\neq\beta$.
We note that Rules \ref{rule:straightforward cross}, \ref{rule:macro cross}, and \ref{rule:micro cross} cover all possible combinations of endpoints for $\alpha$ and $\beta$.
We show that if $\mathfrak{c}(\alpha,\beta)=1$ then $\{M_{|a,b|},M_{|c,d|}\}$ is not $\mathbf E$-compatible and if $\mathfrak{c}(\alpha,\beta)=0$ then $\{M_{|a,b|},M_{|c,d|}\}$ is $\mathbf E$-compatible.
We follow the order in which the rules were stated.

\underline{Rule \ref{rule:straightforward cross}.}
If the endpoints of $\alpha$ and $\beta$ are all contained in $\Edown$ or all contained in $\Eup$ then our if and only if statement follows from arguments similar to those in the proof of Lemma \ref{lem:crossing = incompatible}.
Without loss of generality, suppose $\alpha$ has endpoints in $\Edown$ and $\beta$ has endpoints in $\Eup$.
Then $\mathfrak{c}(\alpha,\beta)=0$.
If $a=s_m$ then $a$ is a source and if $b=s_n$ then $b$ is a sink.
Dual statements for $c$ and $d$ are true as well.
Using Definition \ref{def:E-cluster} and Proposition \ref{prop:incompatible rectangles} we see that $\{M_{|a,b|}, M_{|c,d|}\}$ is $\mathbf E$-compatible.

\underline{Rule \ref{rule:macro cross}.}
Suppose $\alpha$ has both endpoints in $\Edown$ and $\beta$ has one endpoint each in $\Edown$ and $\Eup$.
For now we assume all four endpoints of $\alpha$ and $\beta$ are distinct.
Suppose $\overline{x} < \overline{y}$,  $\overline{z}\in \Edown$, and $\overline{w}\in\Eup$.
If $\overline{x} <\overline{z} < \overline{y}$ then $\mathfrak{c}(\alpha,\beta)=1$ and one verifies there exists a distinguished triangle
\begin{displaymath}
M_{|a,b|} \to M_{|a,d|}\oplus M_{|c,b|}\to M_{|c,d|}\to 
\end{displaymath}
in $\CARS$.
By Proposition \ref{prop:incompatible rectangles}, $\{M_{|a,b|},M_{|c,d|}\}$ is not $\mathbf E$-compatible.
If $\overline{x}<\overline{z}<\overline{y}$ then $\mathfrak{c}(\alpha,\beta)=1$ and one verifies there exists a distinguished triangle
\begin{displaymath}
M_{|c,d|}\to M_{|c,b|}\oplus M_{|a,d|}\to M_{|a,b|} \to
\end{displaymath}
in $\CARS$ and by the same proposition $\{M_{|a,b|},M_{|c,d|}\}$ is not $\mathbf E$-compatible.
If $\overline{z} <\overline{x}$ or $\overline{y} < \overline{z}$ we know $\mathfrak{c}(\alpha,\beta)=0$ and it is straightforward to check that the $g$-vectors of $M_{|a,b|}$ and $M_{|c,d|}$ are $\mathbf E$-compatible.
Thus, $\{M_{|a,b|}, M_{|c,d|}\}$ is $\mathbf E$-compatible.

Now we check when $\alpha$ and $\beta$ each have one endpoint in $\Edown$ and the other in $\Eup$.
Suppose $\mathfrak{c}(\alpha,\beta)=1$.
For Rule~\ref{rule:macro cross}(\ref{itm:up and down}), and without loss of generality, let $\overline{x},\overline{z}\in\Edown$ and $\overline{y},\overline{w}\in\Eup$.
Up to symmetry, we have $\overline{x}<\overline{z}<\overline{y}<\overline{w}$ and so $a<c<b<d$ in $\R$.
One then verifies there exists a distinguished triangle
\begin{displaymath}
M_{|a,b|}\to M_{|a,d|}\oplus M_{|c,b|} \to M_{|c,d|}\to
\end{displaymath}
in $\CARS$.
Again using Proposition \ref{prop:incompatible rectangles} we see $\{M_{|a,b|},M_{|c,d|}\}$ is not $\mathbf E$-compatible.

For Rule~\ref{rule:macro cross}(\ref{itm:up and down 2}), and without loss of generality, let $\overline{x},\overline{w}\in\Edown$ and $\overline{y},\overline{z}\in\Eup$.
Then $\overline{x}<\overline{w}$ and $\overline{z}<\overline{y}$ and
one verifies there exists a distinguished triangle
\begin{displaymath}
M_{|c,d|}\to M_{|c,b|}\oplus M_{|a,d}\to M_{|a,b|}\to
\end{displaymath}
in $\CARS$.
Again, $\{M_{|a,b|},M_{|c,d|}\}$ is not $\mathbf E$-compatible.

Now suppose $\mathfrak{c}(\alpha,\beta)=0$.
If $\overline{x} > \overline{w}$ or $\overline{z} > \overline{y}$, one verifies the $g$-vectors of $M_{|a,b|}$ and $M_{|c,d|}$ are $\mathbf E$-compatible.
If $\overline{x} < \overline{w}$ and $\overline{z}<\overline{y}$ then, up to symmetry $\overline{x},\overline{z}\in\Edown$ and $\overline{y},\overline{w}\in\Eup$.
This means that $\overline{x}<\overline{z}$ and $\overline{y}\overline{w}$ or that $\overline{z}<\overline{x}$ and $\overline{w}<\overline{y}$.
Again one my check the $g$-vectors to see that $M_{|a,b|}$ and $M_{|c,d|}$ are $\mathbf E$-compatible.

\underline{Rule \ref{rule:micro cross}.}
Now we assume $\alpha$ and $\beta$ share an endpoint.

If $\overline{x}=\overline{z}$, then a straightforward calculation shows the $g$-vectors of $M_{|a,b|}$ and $M_{|c,d|}$ are $\mathbf E$-compatible.
Symmetrically, if $\overline{y}=\overline{w}$, then $\{M_{|a,b|},M_{|c,d|}\}$ is $\mathbf E$-compatible.

Next suppose $\overline{x}=\overline{w}=(e,\e)$, for $e\in \R$.
Then $M_{|a,b|}=M_{|e,b|}$ and $M_{|c,d|}=M_{|c,e|}$.
In particular, $e\in |e,b|$ if and only if $e\notin |c,e|$.
Then one verifies the following is a distinguished triangle in $\CAR$:
\begin{displaymath} M_{|c,d|}\to M_{|c,b|}\to M_{|a,b|}\to. \end{displaymath}
By Proposition \ref{prop:incompatible rectangles} again we see $\{M_{|a,b|},M_{|c,d|}\}$ is not $\mathbf E$-compatible.

Finally, suppose $\overline{x}=\overline{w} = ((s_n,s_{n+1}),\e)$ and note that $\overline{y}\neq\overline{z}$.
If $\e=-$ then $|a,b|=[s_n,b|$ and $|c,d|=|c,s_n)$.
If $\e=+$ then $|a,b|=(s_{n+1},b|$ and $|c,d|=|c,s_{n+1}]$. 
In either case, one verifies we have the following distinguished triangle in $\CARS$:
\begin{displaymath} M_{|c,d|}\to M_{|c,b|}\to M_{|a,b|}\to. \end{displaymath}
By Proposition \ref{prop:incompatible rectangles} again we see $\{M_{|a,b|},M_{|c,d|}\}$ is not $\mathbf E$-compatible.

\ul{Conclusion.} For each of Rules \ref{rule:straightforward cross}, \ref{rule:macro cross}, and \ref{rule:micro cross} we have shown (i) if $\mathfrak c(\alpha,\beta)=1$ then $\{M_{|a,b|},M_{|c,d|}\}$ is not $\mathbf E$-compatible and (ii) if $\mathfrak c(\alpha,\beta)=0$ then $\{M_{|a,b|},M_{|c,d|}\}$ is $\mathbf E$-compatible.
\end{proof}

\begin{definition}\label{def:general arc category}
	Let $\CsARS$ be an additive category whose indecomposable objects are $\mathcal A$.
	Define Hom spaces and composition of morphisms similarly to \ref{def:arc indecomposable bijection}.
	This also yields a Krull--Schmidt category.
	We say $\{\alpha,\beta\}$ is $\NRSclosed$-compatible if and only if $\alpha=\beta$ or $\mathfrak{c}(\alpha,\beta)=0$.
\end{definition}

\begin{remark}
	Notice $\NRSclosed$-compatible is equivalent to $\Hom$-orthogonal, not $\Ext$-orthogonal.
\end{remark}

\begin{cor}[to Lemma \ref{lem:general crossing = incompatible}]\label{cor:general crossing = incompatible}
Let $M_{|a,b|}$ and $M_{|c,d|}$ be indecomposables in $\CAR$.
Then $\{\Phi(M_{|a,b|}),\Phi(M_{|c,d|})\}$ is $\NRclosed$-compatible if and only if $\{M_{|a,b|},M_{|c,d|}\}$ is $\mathbf E$-compatible.
\end{cor}

\begin{theorem}\label{thm:general geometric model}
Let $\ARS$ be a continuous quiver of type $A$.
The pairwise compatibility condition $\NRSclosed$ induces the $\NRSclosed$-cluster theory of $\CsARS$ and $\Phi$ induces an isomorphism of cluster theories $(F,\eta):\mathscr T_{\NRSclosed}(\CsARS)\to \mathscr T_{\mathbf E}(\CARS)$.
\end{theorem}
\begin{proof}
By Proposition \ref{prop:all bijection} and Definition \ref{def:general arc category} we have a bijection $\Phi:\Ind(\CARS)\to \Ind(\CsARS)$.
The set $\{M_{|a,b|},M_{|c,d|}\}$ is $\mathbf E$-compatible if and only if $\{\Phi(M_{|a,b|}), \Phi(M_{|c,d|})\}$ is $\NRSclosed$-compatible, by Corollary \ref{cor:general crossing = incompatible}.
Thus by Lemma \ref{lem:bijection and same compatibility gives equivalence} $\NRSclosed$ induces the cluster theory $\mathscr T_{\NRSclosed}(\CsARS)$ and we have the isomorphism of cluster theories given by $F(T) := \{\Phi^{-1}(\alpha) \mid \alpha\in T \}$ and $\eta_T(\alpha) := \Phi^{-1}(\alpha)$.
\end{proof}

\subsection{On the Classification of Cluster Theories of Continuous Type $A$}\label{sec:cluster theory classification}
In this section we identity some cluster theories of continuous type $A$ which are isomorphic.
We show there are at least four isomorphism classes of such cluster theories.
The following notation is useful.
\begin{notation}\label{note:isomorphism of cluster theories}
Let $\mathscr T_{\mathbf P}(\mathcal C)$ and $\mathscr T_{\mathbf Q}(\mathcal D)$ be two cluster theories.
If there is an isomorphism of cluster theories $(F,\eta):\mathscr T_{\mathbf P}(\mathcal C)\to \mathscr T_{\mathbf Q}(\mathcal D)$ then we say $\mathscr T_{\mathbf P}(\mathcal C)$ is isomorphic to $\mathscr T_{\mathbf Q}(\mathcal D)$ and write $\mathscr T_{\mathbf P}(\mathcal C)\cong \mathscr T_{\mathbf Q}(\mathcal D)$.
\end{notation}

\begin{notation}\label{note:reversing notation}
	Let $\ARS$ be a continuous quiver of type $A$.
	\begin{itemize}
		\item By $(\ARS)^{-1}$ we denote the continuous quiver $A_{\R,R}$ where, if $-\infty\neq s_0$, each source $r_n$ in $A_{\R,R}$ is equal to a sink $s_{n-1}$ and similarly for sinks in $R$.
			If $-\infty=s_0$ in $\ARS$, then each source $r_n$ in $A_{\R,R}$ is instead equal to a source $r_{n+1}$ and similarly for sinks in $R$.
		\item By $-(\ARS)$ we denote the the continuous quiver $A_{\R,R}$ where each sink $r_{2n}$ in $A_{\R,R}$ is equal to the sink $-s_{-2n}$ in $\ARS$ and similarly for sources.
	\end{itemize}	
\end{notation}

\begin{remark}\label{rmk:reversing remark}
	Notice that $-(-(\ARS))=\ARS$.
	Furthermore, if $-\infty=s_0$ or $-\infty=s_{-1}$ then $((\ARS)^{-1})^{-1})=\ARS$.
	If $-\infty\neq s_0$ and $-\infty\neq s_1$, then we still have $\repARS$ is equivalent to $\rep_{\kk}(((\ARS)^{-1})^{-1})$ as $\kk$-linear abelian categories.
	
	Finally, we see $(-(\ARS))^{-1}=((-((\ARS)^{-1}))^{-1})^{-1}$ and so $\rep_{\kk}((-(\ARS))^{-1})$ is equivalent to $\rep_{\kk}(-((\ARS)^{-1}))$ as abelian categories.
\end{remark}

\begin{proposition}\label{prop:opposite cluster theories}
	Let $\ARS$ be a continuous quiver of type $A$ and $A_{\R,R}=(\ARS)^{-1}$.
	Then $\mathscr T_{\NRSclosed}(\CsARS)\cong \mathscr T_{\mathbf N_{\R,R}} (\mathcal C_{\overline{\R,R}})$.
\end{proposition}
\begin{proof}
	Denote the sets of endpoints and arcs for $\ARS$ by $\mathcal E_S$ and $\mathcal A_S$.
	Denote the sets of endpoints and arcs for $A_{\R,R}$ by $\mathcal E_R$ and $\mathcal A_R$.
	Let the respective crossing functions be $\mathfrak{c}_S$ and $\mathfrak{c}_R$.
	There is an order preserving bijections $g:\Edown_S \stackrel{\cong}{\to} \Eup_R$ and $h:\Eup_S\stackrel{\cong}{\to} \Edown_R$.
	Let $f:\mathcal E_S\stackrel{\cong}{\to} \mathcal E_R$ be the bijection that is $g$ on $\Edown_S$ and $h$ on $\Eup_S$.
	
	Notice Rules \ref{rule:easy cross}, \ref{rule:macro cross}, and \ref{rule:micro cross} are symmetric with respect to $\Eup$ and $\Edown$, except at sinks and sources.
	Let $\alpha=(\overline{x},\overline{y})$ and $\beta=(\overline{z},\overline{w})$ be in $\mathcal A_S$.
	Let $\gamma=(f(\overline{x},f(\overline{y}))$ and $\delta=(f(\overline{z},\overline{w}))$.
	Then $\mathfrak{c}_S(\alpha,\beta)=1$ if and only if $\mathfrak{c}_R(\gamma,\delta)=1$.
	Now apply Lemma \ref{lem:bijection and same compatibility gives equivalence}.
\end{proof}

\begin{proposition}\label{prop:reversing cluster theories}
	Let $\ARS$ be a continuous quiver of type $A$ and $A_{\R,R}=-(\ARS)$.
	Then $\mathscr T_{\NRSclosed}(\CsARS)\cong \mathscr T_{\mathbf N_{\R,R}} (\mathcal C_{\overline{\R,R}})$.
\end{proposition}
\begin{proof}
	Let $\mathcal E_S$, $\mathcal A_S$, $\mathfrak{c}_s$, $\mathcal E_R$, $\mathcal A_R$, and $\mathfrak{c}_R$ be as in the proof of Proposition~\ref{prop:opposite cluster theories}.
	Then we have order \emph{reversing} bijections $g:\Edown_S \stackrel{\to}{\cong} \Edown_R$ and $h:\Eup_S\to\Eup_R$.
	Let $f:\mathcal E_S\stackrel{\to}{\cong}\mathcal E_R$ be the bijection that is $g$ on $\Edown_S$ and $h$ on $\Eup_S$.
	Now proceed by a similar argument to Proposition~\ref{prop:opposite cluster theories}.
\end{proof}

\begin{theorem}\label{thm:reverse opposite cluster theories}
	Let $\ARS$ be a continuous quiver of type $A$.
	Then there is a diagram of isomorphisms of cluster theories:
	\begin{displaymath}
		\xymatrix{
			\mathscr T_{\mathbf E}(\CARS) \ar@{<->}[r]^-{\cong} \ar@{<->}[d]_-{\cong} & \mathscr T_{\mathbf E}(\mathcal C((\ARS)^{-1})) \ar@{<->}[d]^-{\cong} \\
			\mathscr T_{\mathbf E}(\mathcal C(-(\ARS))) \ar@{<->}[r]_-{\cong} & \mathscr T_{\mathbf E}(\mathcal C(-((\ARS)^{-1}))).
		}
	\end{displaymath}
\end{theorem}
\begin{proof}
	Apply Propositions \ref{prop:opposite cluster theories} and \ref{prop:reversing cluster theories} and Remark \ref{rmk:reversing remark}.
\end{proof}
\begin{cor}\label{cor:sinks and sources count}
	Let $\ARS$ and $A_{\R,S}$ be continuous quivers of type $A$ such that one of the following is true: (i) $|S|=|R|$ and $|S|<\infty$, (ii) $S$ and $R$ are both bounded on exactly one side, or (iii) both $S$ and $R$ are indexed by $\Z$.
	Then $\mathscr{T}_{\mathbf{E}}(\CARS)\cong \mathscr{T}_{\mathbf E}(\mathcal C(A_{\R,R}))$.
\end{cor}
\begin{proof}
	If there is an order- \emph{and} indexing-preserving bijection $S\cong R$ then $\repARS$ and $\rep_{\kk}(A_{\R,R})$ are equivalent as abelian categories and $\mathscr{T}_{\mathbf{E}}(\CARS)\cong \mathscr{T}_{\mathbf E}(\mathcal C(A_{\R,R}))$.
	Now apply Theorem \ref{thm:reverse opposite cluster theories}.
\end{proof}

The classification of cluster theories in Corollary \ref{cor:sinks and sources count} is nearly the classification of derived categories in Theorem \ref{thm:derived equivalence}.

We have two remaining isomorphsisms of cluster theories we would like:
\begin{enumerate}
\item Any isomorphism between $\mathscr T_{\NRSclosed}(\CsARS)$ and $\mathscr T_{\mathbf N_{\overline{\R,R}}}(\mathcal C_{\overline{\R,R}})$ where $\ARS$ has an even number $\geq 2$ of sinks and sources in $\R$ and $A_{\R,R}$ has an odd number of sinks and sources in $\R$.
\item An isomorphism between $\mathscr T_{\NRclosed}(\CsAR)$ and $\mathscr T_{\NRSclosed}(\CsARS)$ where $\AR$ has no sinks or sources in $\R$ and $\ARS$ has an even number $\geq 2$ of sinks and sources in $\R$.
\end{enumerate}

We immediately share the unfortunate news:
\begin{proposition}
Let $\AR$ be a continuous quiver of type $A$ with straight descending or straight ascending orientation.
Let $\ARS$ be a continuous quiver of type $A$ with at least one sink or source in $\R$.
Then there is no isomorphism of cluster theories $\mathscr T_{\NRclosed}(\CsAR)\to \mathscr T_{\NRSclosed}(\CsARS)$.
\end{proposition}
\begin{proof}
The arc $\alpha$ corresponding to the indecomposable $M_{(-\infty,+\infty)}$ in $\CAR$ is in every $\NRclosed$-cluster of $\mathscr T_{\NRclosed}(\CsAR)$.
The arcs corresponding to the projectives from $\repARS$ form an $\NRSclosed$-cluster; this is similarly true for the arcs corresponding to the injectives from $\repARS$.
However, there are not projective-injective objects in $\repARS$ and so these two clusters share no elements.
Therefore, there cannot be such an isomorphism of cluster theories.
\end{proof}

This leaves us with at least four isomorphism classes of cluster theories of continuous type $A$: (i) no sinks or sources in $\R$, (ii) finitely-many sinks and sources in $\R$, (iii) half-bounded sinks and sources in $\R$, and (iv) unbounded sinks and sources in $\R$.
However, it is not clear whether (ii) is just one class, separate classes for even and odd numbers, or a separate class for all numbers.

\noindent \textbf{Open Questions:} 
\begin{itemize}
\item Does there exist a weak equivalence of cluster theories 
	\begin{displaymath}
		\mathscr T_{\NRclosed}(\CsAR)\to \mathscr T_{\NRSclosed}(\CsARS) \qquad \text{or} \qquad \mathscr T_{\NRSclosed}(\CsARS)\to \mathscr T_{\NRclosed}(\CsAR),
	\end{displaymath}
	where $\AR$ has no sinks or sources in $\R$ and $\ARS$ has an even number $\geq 2$ of sinks and sources in $\R$?
\item Does there exist an isomorphism of cluster theories or weak equivalence of cluster theories 
	\begin{displaymath}
		\mathscr T_{\NRSclosed}(\CsARS)\to \mathscr T_{\mathbf N_{\overline{\R,R}}}(\mathcal C_{\overline{\R,R}}) \qquad \text{or} \qquad \mathscr T_{\mathbf N_{\overline{\R,R}}}(\mathcal C_{\overline{\R,R}})\to \mathscr T_{\NRSclosed}(\CsARS),
	\end{displaymath}
	where $\ARS$ has an odd number $n$ of sinks and sources in $\R$ and $A_{\R,R}$ has $n+1$ sinks and sources in $\R$?
\end{itemize}

\subsection{Connection to $\mathbf E$-Mutations}\label{sec:connections to E-mutations}
Let $\ARS$ be a continuous quiver of type $A$.
In this section we use geometric models to draw a $\NRSclosed$-mutation corresponding to an $\mathbf E$-mutation.
Because of our rules on crossing, mutation is not as clearly described as swapping diagonals of a quadrilateral.
However, we can make similar descriptions.
Let us begin with the ``microscopic'' scale.
Let $\ARS$ be a continuous quiver of type $A$ with at least one sink or source in $\R$.
Let $a<b\in\R$ such that neither $a$ nor $b$ is a sink or source and $(a,\e),(b,\e)\in\Edown$, for any $\e\in\{-,+\}$.

Let $T$ be an $\NRSclosed$-cluster such that $((a,-),(b,+)),((a,+), (b,+)),((a,+),(b,-))\in T$.
These correspond to the indecomposables $M_{[a,b]}$, $M_{(a,b]}$, and $M_{(a,b)}$, respectively, in $\CARS$.
We can mutate at $\{(a,+), (b,+)\}$ to obtain $(T\setminus\{\{(a,+),(b,+)\}\})\cup\{\{(a,-),(b,-)\}\}$.
The picture one should have in mind is Figure \ref{fig:micro mutation}.
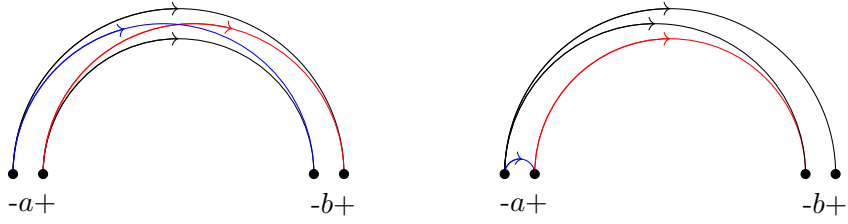
\begin{figure}
\centering
\begin{tikzpicture}[scale=2]
\draw (-0.1,0) arc (180:0:1.1);
\draw[->] (-0.1,0) arc (180:90:1.1);
\draw (0.1,0) arc (180:0:0.9);
\draw[->] (0.1,0) arc (180:90:0.9);
\draw[red] (0.1,0) arc (180:0:1);
\draw[red, ->] (0.1,0) arc (180:75:1);
\draw[blue] (-0.1,0) arc (180:0:1);
\draw[blue, ->] (-0.1,0) arc (180:105:1);
\foreach \x in {0,2}
{
	\filldraw (\x-0.1,0) circle[radius=.3mm];
	\filldraw (\x+0.1,0) circle[radius=.3mm];
}
\draw (0.03,-0.2) node {-$a$+};
\draw (2.03,-0.2) node {-$b$+};
\end{tikzpicture}
\qquad\qquad
\begin{tikzpicture}[scale=2]
\draw (-0.1,0) arc (180:0:1.1);
\draw[->] (-0.1,0) arc (180:90:1.1);
\draw[red] (0.1,0) arc (180:0:0.9);
\draw[->, red] (0.1,0) arc (180:90:0.9);
\draw (-0.1,0) arc (180:0:1);
\draw[->] (-0.1,0) arc (180:90:1);
\draw[blue] (-0.1,0) arc (180:0:0.1);
\draw[blue,->] (-0.1,0) arc (180:80:0.1);
\foreach \x in {0,2}
{
	\filldraw (\x-0.1,0) circle[radius=.3mm];
	\filldraw (\x+0.1,0) circle[radius=.3mm];
}
\draw (0.03,-0.2) node {-$a$+};
\draw (2.03,-0.2) node {-$b$+};
\end{tikzpicture}
\caption{Depictions of mutation on the ``microscopic'' scale, where we replace the red arc with the blue arc. First we mutate at $((a,+), (b,+))$ and then mutate at $((a,+),(b,-))$. Notice the arc orientations pointing from the lower element to the upper element. However, we only \emph{need} these orientations in the right picture.}\label{fig:micro mutation}
\end{figure}

We now move to the ``macroscopic'' scale.
In $\CARS$, we know that if $\{M_{|a,b|}M_{|c,d|}\}$ is not $\mathbf E$-compatible then, up to reversing the roles of the indecomposables, we have the following distinguished triangle in $\CARS$:
\begin{displaymath} M_{|a,b|} \to M_{|a,d|}\oplus M_{|c,b|}\to M_{|c,d|} \to\end{displaymath}
where one of $M_{|a,d|}$ or $M_{|c,b|}$ may be 0.
Now suppose we are $\mathbf E$-mutating in some cluster at $M_{|a,b|}$ and obtain $M_{|c,d|}$, the left picture in Figure \ref{fig:macro mutation}.
If the middle object in the distinguished triangle is not an indecomposable, then, in the geometric model, we have two of the four sides of the quadrilateral we see in triangulations of polygons.

However, we do not know if we have the indecomposables corresponding to the two dotted arcs that complete the quadrilateral.
The dotted arcs may be incompatible with $M_{|a,b|}$ and/or $M_{|c,d|}$.
For example, if $b\in|a,b|$ then there is no arc with $(b,-)$ as a lower endpoint that is compatible with either of the arcs corresponding to $M_{|a,b|}$ and $M_{|c,b|}$.
In the case where one of $M_{|c,b|}$ or $M_{|a,d|}$ is 0, we instead have the right picture in Figure~\ref{fig:macro mutation}.
If some of the endpoints are in $\Edown$ and others in $\Eup$ then we instead draw pictures such as those in Figure \ref{fig:macro mixed mutation}.
\begin{figure}[h]
\centering
\begin{tikzpicture}[scale=2];
\draw (0,0) arc (180:0:1.5);
\draw[->] (0,0) arc (180:90:1.5);
\draw (1,0) arc (180:0:0.5);
\draw[->] (1,0) arc (180:90:0.5);
\draw[red] (0,0) arc (180:0:1);
\draw[->, red] (0,0) arc (180:90:1);
\draw[blue] (1,0) arc (180:0:1);
\draw[->, blue] (1,0) arc (180:90:1);
\draw[dotted, thick](0,0) arc (180:0:0.5);
\draw[dotted, thick ](2,0) arc (180:0:0.5);
\foreach \x in {0,...,3}
	\filldraw (\x,0) circle[radius=.3mm];
\end{tikzpicture}
\qquad \qquad
\begin{tikzpicture}[scale = 2]
\draw (0,0) arc (180:0:1);
\draw[->] (0,0) arc (180:90:1);
\draw[red] (0,0) arc (180:0:0.5);
\draw[->, red] (0,0) arc (180:90:0.5);
\draw[blue] (1,0) arc (180:0:0.5);
\draw[->, blue] (1,0) arc (180:90:0.5);
\foreach \x in {0,...,2}
	\filldraw (\x,0) circle[radius=.3mm];
\end{tikzpicture}
\caption{Depictions of mutation on the ``macroscopic'' scale where we replace red arcs with blue arcs. Notice the arc orientations pointing from the lower element to the upper element. We only \emph{need} the orientations in the right picture.}\label{fig:macro mutation}
\end{figure}
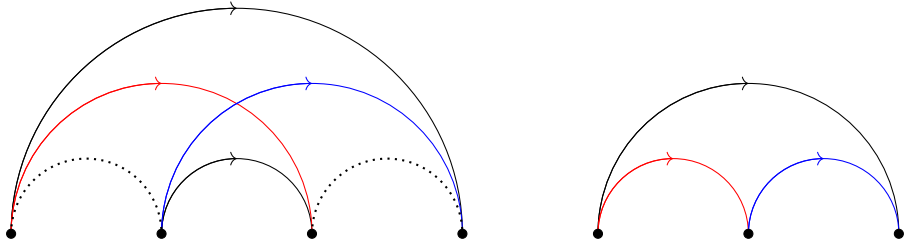
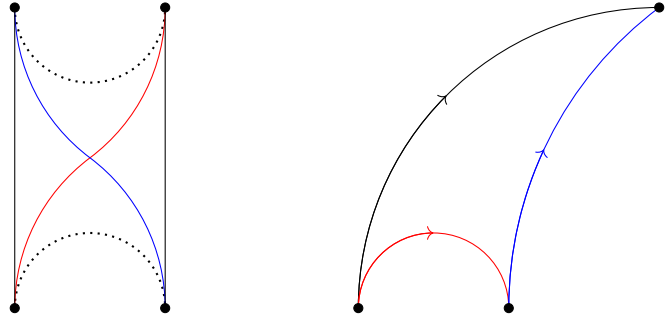
\begin{figure}[h]
\centering	
\begin{tikzpicture}[scale = 2]
\draw[dotted, thick] (0,0) arc (180:0:0.5);
\draw[dotted, thick] (0,2) arc (180:360:0.5);
\draw[red] (0,0) arc (180:126.87:1.25) arc (306.87:360:1.25);
\draw[blue] (0,2) arc (180:233.13:1.25) arc (53.13:0:1.25);
\foreach \x in {0,1}
{
	\foreach \y in {0,2}
	{
		\filldraw (\x,\y) circle[radius=.3mm];
	}
	\draw (\x,0)-- (\x,2);
}
\end{tikzpicture}
\qquad \qquad \qquad
\begin{tikzpicture}[scale = 2]
\draw (0,0) arc (180:90:2);
\draw[->] (0,0) arc (180:135:2);
\draw[red] (0,0) arc (180:0:0.5);
\draw[->, red] (0,0) arc (180:90:0.5);
\draw[blue] (1,0) arc (180:126.87:2.5);
\draw[->, blue] (1,0) arc (180:155:2.5);
\filldraw (0,0) circle[radius=.3mm];
\filldraw (1,0) circle[radius=.3mm];
\filldraw (2,2) circle[radius=.3mm];
\end{tikzpicture}
\caption{Depiction of macroscopic mutation where the arcs have endpoints mixed above and below. We mutate the red arc to the blue arc. The orientations of the arcs in the left picture do not matter but we need these in the right picture.}\label{fig:macro mixed mutation}
\end{figure}

\section{Continuous Mutation and Mutation Paths}\label{sec:continuous mutation and mutation paths}
This section is dedicated to the definition of a continuous mutation and the basic properties of continuous mutations.
These generalize the familiar notion of mutation in a cluster structure.
We define this new type of mutation for all cluster theories (Definition \ref{def:cluster theory}) though we use type $A$ cluster theories for our examples.
Notably, for a cluter theory $\mathscr T_{\mathbf P}(\mathcal C)$, any $\mathbf P$-mutation can be thought of as a continuous $\mathbf P$-mutation (see Example \ref{xmp:traditional mutation}).
In Section \ref{sec:connections to E-mutation paths}, we show how to interpret continuous mutations via geometric models (Sections \ref{sec:straight AR} and \ref{sec:other orientations of AR}) using continuous type $A$ as an example.
The final subsection of this section is dedicated to the space of mutations (Definition \ref{def:space of mutations}), which generalizes the exchange graph of a cluster structure.
We pose questions related specifically to the $\mathbf E$-cluster theory of an arbitrary continuous quiver of type $A$ at the end.

For Section \ref{sec:continuous mutation and mutation paths} we fix $\mathcal C$ a Krull--Schmidt category and $\mathbf P$ a pairwise compatibility condition on the indecomposables in $\mathcal C$ such that $\mathbf P$ induces the $\mathbf P$-cluster theory of $\mathcal C$ (Definition \ref{def:cluster theory}).

\subsection{Continuous Mutation}\label{sec:continuous mutation}
In this section we define continuous mutation.
In order to better interpret continuous mutations, we need to define a trivial mutation.

\begin{definition}\label{def:trivial mutation}
	A \ul{trivial $\mathbf P$-mutation} is an identity function $\id_T:T\to T$, for any $\mathbf P$-cluster $T$.
\end{definition}

\begin{definition}\label{def:continuous mutation}
	Let $T$ and $T'$ be $\mathbf P$-clusters and $\mu:T\to T'$ a bijection.
	We call $\mu$ a \ul{continuous $\mathbf P$-mutation} if it satisfies the following four properties.
	\begin{itemize}
		\item There is a set $S\subset T$ such that $\mu X = X$ if and only if $X\notin S$.
		\item Let $S'=\mu (S)$. For all $\mu X\in S'$, (i) $\mu X\notin T$ and (ii) $\{X,\mu X\}$ is not $\mathbf P$-compatible.
		\item There exist injections $f_\mu: S\to [0,1]$ and $g_\mu:S'\to [0,1]$ such that $(g_\mu \circ \mu)|_S = (f_\mu)|_S$.
		\item For any subinterval $J\subset [0,1]$, where $0\in J$ and $1\notin J$, the following is a $\mathbf P$-cluster:
			\begin{displaymath}
				\left(T\setminus f_\mu^{-1}(J)\right) \cup g_\mu^{-1}(J) = \left(T'\setminus g_\mu^{-1}(\bar{J})\right) \cup f_\mu(\bar{J}),
			\end{displaymath}
			where $\bar{J}=[0,1]\setminus J$.
	\end{itemize}
\end{definition}

We need to justify the word `mutation.'
We do this with Propositions \ref{prop:continuous mutations are reversible} and \ref{prop:a continuum of mutations}.
The first states that every continuous mutation can be reversed.
The second states that we may consider a continuous mutation a collection of mutations, one each at time $t$, for all $t\in[0,1]$.

The proof of the following proposition is a straightforward application of the definition.
\begin{proposition}\label{prop:continuous mutations are reversible}
	Let $\mu:T\to T'$ be a continuous $\mathbf P$-mutation.
	Then $\mu^{-1}:T'\to T$ is also a continuous $\mathbf P$-mutation.
\end{proposition}

\begin{proposition}\label{prop:a continuum of mutations}
	Let $\mu:T\to T'$ be a continuous $\mathbf P$-mutation. 
	For every $t\in [0,1]$, the following bijection $(T\setminus f_\mu^{-1}([0,t)))\cup g_\mu^{-1}([0,t))  \to (T\setminus f_\mu^{-1}([0,t]))\cup g_\mu^{-1}([0,t])$ is a $\mathbf P$-mutation:
	\begin{displaymath}
	X\mapsto \left\{ \begin{array}{ll} X & X\neq f^{-1}(t) \\ g^{-1}(t) & X=f^{-1}(t). \end{array} \right.
	\end{displaymath}
\end{proposition}
\begin{proof}
	In the case $(T\setminus f_\mu^{-1}([0,t)))\cup g_\mu^{-1}([0,t)) = (T\setminus f_\mu^{-1}([0,t]))\cup g_\mu^{-1}([0,t])$ we have a trivial $\mathbf P$-mutation.
	Suppose $(T\setminus f_\mu^{-1}([0,t)))\cup g_\mu^{-1}([0,t))  \neq (T\setminus f_\mu^{-1}([0,t]))\cup g_\mu^{-1}([0,t])$.
	Since $f_\mu$ and $g_\mu$ are injections, $f_\mu^{-1}([0,t))$ differs from $f_\mu^{-1}([0,t])$ by at most one element and by assumption they differ by at least one element; thus differing by exactly one element.
	This is similarly true for $g_\mu^{-1}([0,t))$ and $g_\mu^{-1}([0,t])$. 
	By definition, $\mu (f_\mu^{-1}(t)) = g_\mu^{-1}(t)$ and $\{f_\mu^{-1}(t), g_\mu^{-1}(t)\}$ are not $\mathbf P$-compatible.
	Therefore, we have a $\mathbf P$-mutation.
\end{proof}

We conclude with this final definition that is useful in asking questions about the classification of $\mathbf E$-clusters in $\CARS$ (Definition \ref{def:E-cluster}) at the end of Section \ref{def:space of mutations}.

\begin{definition}\label{def:sequence of continuous mutations}
	Let $Z=\{1,\ldots,n\}$ or $Z=\Z_{>0}$.
	For each $i\in Z$ let $\mu_i$ be a continuous $\mathbf P$-mutation such that the target of $\mu_i$ is the source of $\mu_{i+1}$ when $i,i+1\in Z$.
	We call $\{\mu_i\}_{i\in Z}$ a \ul{sequence of continuous $\mathbf P$-mutations}.
	If each $\mu_i$ mutates only one element of $T_i$ we may also say that $\{\mu_i\}$ is a \ul{sequence of $\mathbf P$-mutations}.
\end{definition}

\subsection{Examples}\label{sec:continuous examples}

In this section we highlight two existing examples of continuous mutations that do not feel so continuous followed by a new example.
The first (Example \ref{xmp:traditional mutation}) shows that a mutation, in the traditional sense, can be thought of as a continuous mutation.
The second (Example \ref{xmp:transfinite mutation}) describes an infinite sequence of mutations.
While these both exist in the literature, the contribution is that continuous mutation unifies the way to describe these types of mutations.
We conclude with Proposition \ref{prop:proj to inj}, which, as far as the author knows, does not exist anywhere in the literature.

\begin{example}\label{xmp:traditional mutation}
Let $\mathcal C$ be a Krull--Schmidt category with pairwise compatibility condition $\mathbf P$ on indecomposables such that $\mathbf P$ induces the $\mathbf P$-cluster theory of $\mathcal C$.
Let $\mu:T\to (T\setminus\{X\})\cup\{Y\}$ be a $\mathbf P$-mutation.
Furthermore, let $S=\{X\}$, $S'=\{Y\}$, and $T'=(T\setminus \{X\})\cup\{Y\}$.
Finally, let $f:\{X\}\to [0,1]$ and $g:\{Y\}\to [0,1]$ each send $X$ and $Y$ to $\frac{1}{2}$, respectively.
This meets the requirements for the definition of a continuous mutation.
\end{example}

The second example is based on the completed infinity-gon from \cite{BaurGraz}. 

\begin{definition}\label{def:completed infinity-gon}
	Let $\mathcal E = \Z\cup\{-\infty,+\infty\}$ with the usual total ordering.
	Let 
	\begin{displaymath}
		\mathcal A = \{(i,j)\in \mathcal E\times\mathcal E \mid \exists k\in\mathcal E \text{ s.t } i<k<j\}\setminus \{(-\infty,+\infty)\}.
	\end{displaymath}
	Define the crossing function $\mathfrak c:\mathcal A\times\mathcal A\to\{0,1\}$ by
	\begin{displaymath}
		\mathfrak c ((i,j),(i',j')) = \begin{cases}
 			1 & ((i,j)=(i',j')) \text{ or } (i<i'<j<j') \text{ or } (i'<i<j<j') \\
 			0 & \text{otherwise}.
 		\end{cases}
	\end{displaymath}
	
	We define $\Cinftyclosed$ to be the additive category whose indecomposable objects are $\mathcal A$.
	Define Hom spaces and composition as in Definitions~\ref{def:arc indecomposable bijection} and \ref{def:general arc category}.
	We again obtain a Krull--Schmidt category.
	For $\alpha\neq\beta$, we say $\{\alpha,\beta\}$ is $\Ninftyclosed$-compatible if and only if $\mathfrak c(\alpha,\beta)=0$.
\end{definition}

Baur and Graz proved in \cite{BaurGraz} that $\Ninftyclosed$ induces the $\Ninftyclosed$-cluster theory of $\Cinftyclosed$.

Baur and Graz define a \ul{$T$-admissible sequence of arcs} $\{\alpha_i\}$ is one where $\alpha_1$ is $\Ninftyclosed$-mutable in $T_1=T$ and each $T_i$ for $i>1$ is obtained by mutating $\alpha_{i-1}$ which must be mutable in $T_{i-1}$.
Note this sequence may be infinite so long as there is a first arc in the sequence.
Baur and Graz note that mutating along a $T$-admissible sequence does not always result in a $\Ninftyclosed$-cluster.
I.e., the colimit of such a sequence of mutations may not be a $\Ninftyclosed$-cluster.

\begin{example}\label{xmp:transfinite mutation}
Let $T$ be an $\Ninftyclosed$-cluster in $\Cinftyclosed$ and $\{\alpha_i\}$ a $T$-admissible sequence of arcs.
Since each $\Ninftyclosed$-mutation $\mu_i:T_i\to T_{i+1}$ is also a continuous $\Ninftyclosed$-mutation any admissible sequence of arcs yields a sequence of continuous $\Ninftyclosed$-mutations.

Now suppose $\{\alpha_i\}\subset T$ and the result of mutating along $\{\alpha_i\}$ yields an $\Ninftyclosed$-cluster $T'$.
Then we let $S=\{\alpha_i\}$ and let $f:S\to [0,1]$ be given by $\alpha_i\mapsto 1-\frac{1}{i+1}$.
Let $S'=\{\mu_i(\alpha_i)\}$ and let $g:S'\to [0,1]$ be given by $\mu_i(\alpha_i)\mapsto 1-\frac{1}{i+1}$.
We now have a continuous $\Ninftyclosed$-mutation. 

In general, a $T$-admissible sequence of arcs can be ``grouped'' into intervals of arcs which each belong to the first cluster of the group.
This yields a sequence of $\Ninftyclosed$-mutations in a somewhat minimal way.
Of course, this does not work if $\{\alpha_i\}\subset T$ and mutation along $\{\alpha_i\}$ does not result in an $\Ninftyclosed$-cluster.
\end{example}

\begin{remark}\label{rem:sequences dont always yield clusters}
Let $\mathcal C$ be a Krull--Schmidt category with pairwise compatibility condition $\mathbf P$ on indecomposables such that $\mathbf P$ induces the $\mathbf P$-cluster theory of $\mathcal C$.
As seen in Example \ref{xmp:transfinite mutation} it might be possible to construct a sequence of (continuous) $\mathbf P$-mutations that does not yield a $\mathbf P$-cluster.
The authors of \cite{BaurGraz} provide a way to complete their compatible sets for their cluster theory.
\end{remark}

\begin{proposition}\label{prop:proj to inj}
Let $\AR$ have the straight descending orientation, $\mathcal Proj$ be the $\mathbf E$-cluster containing all the projectives from $\repAR$, and $\mathcal Inj$ be the $\mathbf E$-cluster containing the injectives from $\repAR$.
There is a sequence of continuous mutations $\{\mu_1, \mu_2\}$ from $\mathcal Proj$ to $\mathcal Inj$.
\end{proposition}
\begin{proof}
Recall that every indecomposable in $\CAR$ comes from an indecomposable $M_I$ in $\repAR$ (Definition \ref{def: MI}, Theorem \ref{thm:indecomp projs}, Proposition \ref{prop:derived is Krull--Schmidt}, and \cite[Proposition 3.1.4]{IgusaRockTodorov2}). Recall also that $|a,b|$ means the inclusion of $a$ or $b$ is either indeterminate or clear from context (see Conventions in the introduction) and Theorem \ref{thm:GeneralizedBarCode}).
Note that $\mathcal Proj\cap \mathcal Inj= \{M_{(-\infty,+\infty)}\}$.

We construct two continuous $\mathbf E$-mutations to mutate $\mathcal Proj$ to $\mathcal Inj$.
First, let $S_1=\mathcal Proj$ and define $f_1:\mathcal Proj\to [0,1]$ in two parts.
For $M_{(-\infty, x)}$ and $M_{(-\infty,x]}$ in $\mathcal Proj$, we let 
\begin{align*}
	f_1\left(M_{(-\infty,x)}\right)&=\frac{1}{2} - \left(\frac{\tan^{-1} x}{2\pi} + \frac{1}{4}\right) &
	f_1\left(M_{(-\infty,x]}\right)&=1 - \left(\frac{\tan^{-1} x}{2\pi} + \frac{3}{4} \right).
\end{align*}
The ``middle'' $\mathbf E$-cluster is 
\begin{displaymath} T_2:=\left\{M_{(-\infty,+\infty)}\right\} \cup \left\{M_{[x,+\infty)},M_{[x,x]} \mid x\in\R\right\}.\end{displaymath} 
We then define $g_1:T_2\to [0,1]$ to match with $f_1$:
\begin{align*}
g_1\left(M_{[x,x]}\right) &=\frac{1}{2} - \left(\frac{\tan^{-1} x}{2\pi} + \frac{1}{4}\right) &
g_1\left(M_{[x,+\infty)}\right) &= 1 - \left(\frac{\tan^{-1} x}{2\pi} + \frac{3}{4} \right).
\end{align*}
Both $f_1$ and $g_1$ are injections and we may define $\mu_1(M) = g^{-1}(f(M))$ and obtain the continuous $\mathbf E$-mutation $\mu_1:\mathcal Proj\to T_2$.

Now let $S_2= \{M_{[x,x]} \mid x\in\R\}\subset T_2$ and $S'_2 = \{M_{(x,+\infty)} \mid x\in\R\}\subset \mathcal Inj$.
We define $f_2:T_2\to [0,1]$ and $g_2:\mathcal Inj\to [0,1]$ by
\begin{displaymath}
f_2 \left( M_{[x,x]} \right) = \frac{\tan^{-1} x}{\pi} + \frac{1}{2} = g_2 \left( M_{(x,+\infty)} \right).
\end{displaymath}
We define $\mu_2 (M)$ to be $M$ if $M\notin S_2$ and $g^{-1}(f(M))$ if $M\in S_2$.
This gives the continuous $\mathbf E$-mutation $\mu_2:T_2\to\mathcal Inj$.
Thus we have a sequence of continuous $\mathbf E$-mutations $\{\mu_1,\mu_2\}$ to mutate the projectives into the injectives.
\end{proof}

\subsection{Mutation Paths}\label{sec:mutation paths}
In this section we define mutation paths, which should be thought of as a generalization of a sequence of mutations.
At first we formally define a long sequence of continuous mutations (Definition \ref{def:lsequence of continuous mutations}) and then move on to mutation paths in general (Definition \ref{def:mutation path}).
Note also that a continuous mutation is an example of a mutation path (Example \ref{xmp:continuous mutation is mutation path}) just as a mutation is an example of a continuous mutation.

A mutation path should be thought of as a generalization of a path of mutations in the exchange graph of a cluster structure.
This is formalized in Section \ref{sec:space of mutations}.
As before, our definitions are for any cluster theory but our interest is in $\mathbf E$-cluster theories of $\AR$ quivers.

\begin{definition}\label{def:lsequence of continuous mutations}
 
Let
\begin{displaymath}
\overline{\mu}=\{{_i\mu} \mid {_iT}_0\to {_i}T_1\}_{i\in \Z}
\end{displaymath}
be a collection of continuous mutations such that ${_iT}_1 = {_{i+1}T}_0$.
This yields a diagram in $\mathcal Sets$:
\begin{displaymath}\xymatrix@C=8ex{
\cdots \ar[r]^-{_{i-1}\mu} & {_{i-1}T}_1 = {_iT}_0 \ar[r]^-{_i\mu} & {_iT}_1 = {_{i+1}T}_0 \ar[r]^-{_{i+1}\mu} & {_{i+1}T}_1 = {_{i+2}T}_0 \ar[r]^-{_{i+2}\mu} & \cdots
}\end{displaymath}
If this diagram has a limit and colimit we call $\overline{\mu}$ a \ul{long sequence of continuous mutations} and we call the limit and colimit the \ul{source} and \ul{target} of $\overline{\mu}$, respectively.
\end{definition}

\begin{definition}\label{def:mutation path}
Define a category $\mathcal I$ whose objects are pairs $(x,i)\in [0,1]\times\{0,1\}$.
Consider $[0,1]$ and $\{0,1\}$ with their respective usual total ordering.
Morphisms in $\mathcal I$ are defined by
\begin{displaymath}
\Hom_{\mathcal I}\left( (s,i), (t,j) \right) :=
\left\{\begin{array}{ll} \{*\} & s<t \text{ or } (s=t\text{ and }i\leq j) \\
\emptyset & \text{otherwise.} \end{array}\right.
\end{displaymath}

Let $\overline{\mu}:\mathcal I\to \mathcal{S}ets$ be a functor such that $\overline{\mu}*:\overline{\mu}(s,0)\to \overline{\mu}(s,1)$ is a (possibly trivial) $\mathbf P$-mutation in $\mathscr T_{\mathbf P}(\mathcal C)$.
Then we call $\overline{\mu}$ a \ul{$\mathbf P$-mutation path}.
\end{definition}

\begin{remark}\label{rmk:mutation path}
The reader may notice that the target of the functor is not $\mathscr T_{\mathbf P}(\mathcal C)$, but just $\mathcal{S}ets$.
This is because we have not defined $\mathscr T_{\mathbf P}(\mathcal C)$ (in Definition \ref{def:cluster theory}) to be closed under any kind of transfinite composition.
However, transfinite composition is indeed sometimes defined in $\mathcal{S}ets$.
For example, if every set in a diagram has the same cardinality and every morphism is a bijection, the transfinite composition is well-defined (and in this case is also a bijection).
We only ensure the smallest morphisms $(s,0)\to (s,1)$ are in $\mathscr T_{\mathbf P}(\mathcal C)$.
\end{remark}

\begin{proposition}\label{prop:inverse path}
 
Let $\overline{\mu}:\mathcal I\to \mathcal Sets$ be $\mathbf P$-mutation path.

Let $\overline{\mu}^{-1}:\mathcal I\to \mathcal Sets$ be a functor given by
\begin{align*}
\overline{\mu}^{-1} (s,i) &:= \overline{\mu} (1-s, 1-i) \\
\overline{\mu}^{-1}( (s_i)\to (t,j) ) &:= \overline{\mu}((1-t,1-j) \to (1-s,1-i)).
\end{align*}
Then $\overline{\mu}^{-1}$ is also a $\mathbf P$-mutation path.
\end{proposition}
\begin{proof}
Since $\mathscr T_{\mathbf P}(\mathcal C)$ is a groupoid inside $\mathcal Sets$ the definition of $\overline{\mu}^{-1}$ amounts to reversing the order of the objects and taking the inverse morphism between each pair of objects in the image.
\end{proof}

\begin{example}\label{xmp:continuous mutation is mutation path}
 
Let $\mu:T\to T'$ be a continuous $\mathbf P$-mutation.
Let $\bar{\mu}:\mathcal I\to \mathcal Sets$ be defined in the following way.
On objects,
\begin{align*} 
\overline{\mu} (s,0) &= (T\setminus f^{-1}[0,s))\cup g^{-1}[0,s) \\ 
\overline{\mu} (s,1) &= (T\setminus f^{-1}[0,s])\cup g^{-1}[0,s].
\end{align*} 
By Proposition \ref{prop:a continuum of mutations}, for each $s\in[0,1]$, $\mu$ defines a $\mathbf P$-mutation $\overline{\mu}(s,0)\to \overline{\mu}(s,1)$.
Define $\overline{\mu}:\overline{\mu}(s,0)\to\overline{\mu}(s,1)$ to be precisely this $\mathbf P$-mutation.
Thus each continuous $\mathbf P$-mutation is a $\mathbf P$-mutation path.
\end{example}

Below we construct some variables $i_s$, $a_s$, $b_s$, and $t_s$ for each $s\in[0,1]$.
We use these to show how a long sequence of continuous mutations can be considered as a mutation path.

\begin{construction}\label{con:sequence of continuous mutations variables}
 
Let $\overline{\mu}$ be a long sequence of continuous mutations and fix $0 < \e << 1$.
For each $s\in(0,1)$, there exists $i \in \Z$ such that
\begin{displaymath}
\left(\frac{\tan^{-1} i}{\pi} + \frac{1}{2} \right) \leq s < \left( \frac{\tan^{-1}(i+1)}{\pi} + \frac{1}{2} \right).
\end{displaymath}
Note that since the right inequality is strict, there is a unique such $i$ for each $s\in(0,1)$.
Denote it by $i_s$.
Let
\begin{align*}
a_s &:= \left( \frac{\tan^{-1} i_s}{\pi} + \frac{1}{2} \right) \\
b_s &:= \left( \frac{\tan^{-1} (i_s+1)}{\pi} + \frac{1}{2} \right).
\end{align*}
Note that if $i_s = i_{s'}$ for $s$ and $s'$ then $a_s=a_{s'}$ and $b_s=b_{s'}$.
We now define $t_s$:
\begin{displaymath}
t_s := \left\{\begin{array}{ll}
0 & s\in [a_s, \, \, (1-\e)a_s + \e b_s] \\
(s-(1-\e)a_s -\e b_s)/ ((1-2\e)(b_s - a_s)) & s\in [(1-\e)a_s + \e b_s,\, \, \e a_s + (1-\e)b_s] \\
1 & s\in (\e a_s + (1-\e)b_s, \, \, b_s).
\end{array}\right.
\end{displaymath}
We provide a picture to make the variables easier to understand for $s\in [\frac{1}{2},\frac{3}{4}]$:
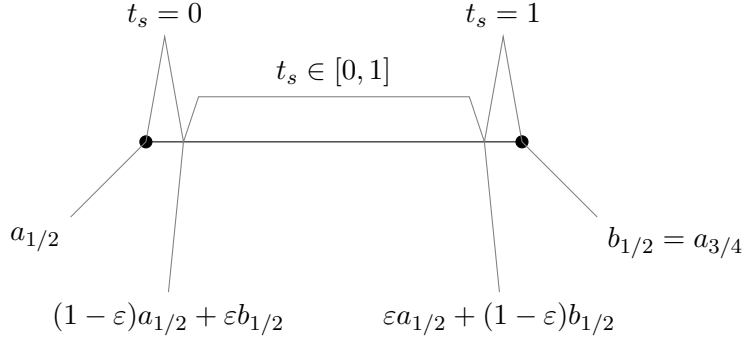
\begin{figure}
\centering
\begin{tikzpicture}[scale = 20]
\draw (.5,0) -- (.75,0);
\filldraw (.5,0) circle[radius=.04mm];
\filldraw (.75,0) circle[radius=.04mm];

\draw[white!50!black] (0.5,0) -- (.45,-.05);
\draw (.45,-.05) node [anchor=north east] {$a_{1/2}$};
\draw[white!50!black] (0.525,0) -- (.515,-.1);
\draw (.515,-.1) node[anchor=north] {$(1-\e)a_{1/2} + \e b_{1/2}$};
\draw[white!50!black] (0.725,0) -- (.735,-.1);
\draw (.735,-.1) node [anchor=north] {$\e a_{1/2} + (1-\e)b_{1/2}$};
\draw[white!50!black] (.75,0) -- (.8,-.05);
\draw (.8,-.05) node [anchor=north west] {$b_{1/2}=a_{3/4}$};

\draw[white!50!black] (.5,0) -- (.5125,.07) -- (.525,0) -- (.535,.03) -- (.715,.03) -- (.725,0) -- (.7375,.07) -- (.75,0);
\draw (.625,.03) node [anchor=south] {$t_s\in [0,1]$};
\draw (0.7375,.07) node[anchor=south] {$t_s=1$};
\draw (0.5125,.07) node[anchor=south] {$t_s=0$};
\end{tikzpicture}
\caption{Schematic of $t_s$ for $s\in[\frac{1}{2},\frac{3}{4}]$.}\label{fig:ts schematic}
\end{figure}
\end{construction}

\begin{proposition}\label{prop:sequence of continuous mutation is a mutation path}
Let $\overline{\mu}$ be a (long) sequence of continuous mutations.
Then $\overline{\mu}$ is also a mutation path.
\end{proposition}
\begin{proof}
We may consider $\overline{\mu}$ as a functor $\mathcal I\to \mathcal Sets$ in the following way.
We now make our assignment on objects:
\begin{align*} 
(s,0) &\mapsto \left\{\begin{array}{ll} {_iT}_0 = {_{i-1}T}_1 & s\in [a_s, (1-\e) a_s + \e b_s ) \\
	( {_{i_s}T}\setminus {_if}^{-1}[0,t_s) )\cup {_ig}^{-1}[0,t_s)  & s\in [(1-\e) a_s + \e b_s, \e a_s + (1-\e)b_s] \\
	{_iT}_1={_{i+1}T}_0 & s\in (\e a_s + (1-\e) b_s, b_s) \end{array}\right. \\ 
(s,1) &\mapsto \left\{\begin{array}{ll} {_iT}_0 = {_{i-1}T}_1 & s\in [a_s, (1-\e) a_s + \e b_s ) \\
	( {_{i_s}T}\setminus {_if}^{-1}[0,t_s] )\cup {_ig}^{-1}[0,t_s] & s\in [(1-\e) a_s + \e b_s, \e a_s + (1-\e)b_s] \\
	{_iT}_1={_{i+1}T}_0 & s\in (\e a_s + (1-\e) b_s, b_s). \end{array}\right. 
\end{align*}
When $s\in [(1-\e) a_s + \e b_s, \e a_s + (1-\e)b_s]$ we see by Proposition \ref{prop:a continuum of mutations} that the morphism $*:(s,0)\to (s,1)$ is sent to a (possibly trivial) $\mathbf P$-mutation. 
When $s\in [a_s, (1-\e) a_s + \e b_s)\cup (\e a_s + (1-\e)b_s, b_s)$ the morphism $*:(s_0)\to (s_1)$ is sent to the trivial $\mathbf P$-mutation on $\bar{\mu}(s,0)$. 
This defines a mutation path.
\end{proof}

The $\e$ ``padding'' in Construction \ref{con:sequence of continuous mutations variables} is necessary to prove Proposition \ref{prop:sequence of continuous mutation is a mutation path}.
If we did not have the ``padding'' we would attempt to assign two $\mathbf P$-mutations, or their composition, to morphisms such as $*:(\frac{3}{4},0)\to(\frac{3}{4},1)$.

\begin{remark}\label{rmk:inverses}
 
Let $\overline{\mu}$ be a long sequence of continuous $\mathbf P$-mutations.
We see in Proposition \ref{prop:sequence of continuous mutation is a mutation path} that for a fixed $\e$ the the inverse \emph{path} $\overline{\mu}^{-1}$ agrees with the inverse \emph{sequence} $\{{_{-i}\mu}\}_{i\in\Z}$.
Thus when working with a long sequence of continuous mutations we need not be specific about which inverse we take as long as an $\e$ has been chosen.
\end{remark}

\begin{definition}\label{def:path composition}
 
Let $\overline{\mu}_1,\overline{\mu}_2:\mathcal I\to \mathcal Sets$ be two $\mathbf P$-mutation paths and suppose $\overline{\mu}_1(1,0) = \overline{\mu}_2(0,0)$ and $\overline{\mu}_1(1,1) = \overline{\mu}_2(0,1)$.

We define the composition of $\mathbf P$-mutation paths, denoted $\overline{\mu}_1\cdot \overline{\mu}_2$ in the following way:
\begin{align*}
\overline{\mu}_1\cdot \overline{\mu}_2(s,i) &:= \left\{ \begin{array}{ll} \overline{\mu}_1(2s,i) & 0\leq s \leq \frac{1}{2} \\ \overline{\mu}_2(2s-1, i) & \frac{1}{2}\leq s \leq 1.  \end{array}\right. \\
\overline{\mu}_1\cdot\overline{\mu}_2 ((s,0)\to (s,1)) &:=  \left\{ \begin{array}{ll} \overline{\mu}_1((2s,0)\to (2s,1)) & 0\leq s \leq \frac{1}{2} \\ \overline{\mu}_2((2s-1, 0)\to (2s-1,1)) & \frac{1}{2}\leq s \leq 1.  \end{array}\right.
\end{align*}
\end{definition}

\begin{proposition}\label{prop:composition is a path}
Let $\overline{\mu}_1$ and $\overline{\mu}_2$ be $\mathbf P$-mutation paths such that
\begin{displaymath}
\overline{\mu}_1 (1,0) = \overline{\mu}_2 (0,0) \qquad \text{and} \qquad \overline{\mu}_1 (1,1) = \overline{\mu}_2(0,1).
\end{displaymath}
Then $\overline{\mu}_1\cdot \overline{\mu}_2$ is a $\mathbf P$-mutation path.
\end{proposition}
\begin{proof}
By assumption the definitions agree at $\frac{1}{2}$.
For $0\leq s < \frac{1}{2}$ and $\frac{1}{2} < t \leq 1$, the morphism $\overline{\mu}_1\cdot \overline{\mu}_2*:\overline{\mu}_1\cdot \overline{\mu}_2(s,i) \to \overline{\mu}_1\cdot \overline{\mu}_2(t,j)$ is the composition
\begin{displaymath} \overline{\mu}_1\cdot \overline{\mu}_2(s,i)\to \overline{\mu}_1\cdot \overline{\mu}_2\left(\frac{1}{2},0\right) \to \overline{\mu}_1\cdot \overline{\mu}_2 \left(\frac{1}{2},1\right) \to \overline{\mu}_1\cdot \overline{\mu}_2(t,j). \qedhere \end{displaymath}
\end{proof}.

\begin{remark}\label{rmk:composition is not a sequence}
The composition of two long sequences of continuous mutations as in Definition \ref{def:path composition} is not a long sequence of continuous mutations as in Proposition \ref{prop:sequence of continuous mutation is a mutation path}.
\end{remark}
%

\subsection{Connection to $\NRSclosed$-mutations}\label{sec:connections to E-mutation paths}
The more interesting pictures of $\NRSclosed$ mutations (Section~\ref{sec:other orientations of AR}) are those of continuous mutations.
In this section we use our geometric models to show how one may picture a continuous $\mathbf E$-mutation by drawing the corresponding continuous $\NRSclosed$-mutation.
In particular, those continuous mutations that cannot be described as any type of sequence of mutations, which is discrete.
Consider $\AR$ with straight descending orientation.
Let $T$ be $\{P_x, M_{[x,x]} \mid x\in\R\}\cup\{P_{+\infty}\}$ and $\phi:\R \to (0,1)$ be some order \emph{reversing} bijection.
Let $f:\{P_x \mid x\in\R\}\to [0,1]$ be given by $P_x\mapsto \phi(x)$.
Let $g:\{I_x \mid x\in\R\}\to [0,1]$ be given by $I_x\mapsto \phi(x)$ and let $T'=\{I_x, M_{[x,x]} \mid x\in \R\}\cup \{P_{+\infty}\}$.
Then we have a continuous mutation $T\to T'$.

We would like to show what this looks like in terms of arcs.
Of course, we cannot depict each of the mutations at time $t$ for all $t\in(0,1)$, as we do not have uncountably-many pages.
However, we can think of the process as an animation and take a few select frames so that we have the general idea.
In Figure~\ref{fig:continuous mutation picture}, we only show 6 frames.
One could make a proper animation at a sufficiently high frame rate to get the full effect.

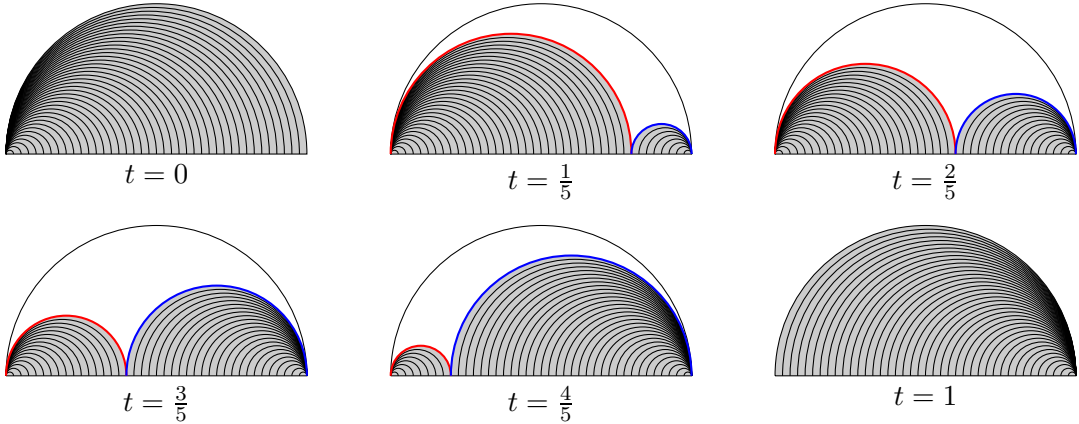
\begin{figure}[h]
\centering
\begin{tikzpicture}
\draw[white] (-0.1,-0.8) -- (4.1,-0.8) -- (4.1,2.1) -- (-0.1,2.1) -- (-0.1,-0.8); 
\filldraw[fill=white!80!black, draw=black] (0,0) arc (180:0:2) -- (0,0);
\foreach \x in {1,...,39}
	\draw (0,0) arc (180:0:0.05*\x);

\draw (2,0) node[anchor=north] {$t=0$};
\end{tikzpicture} \qquad
\begin{tikzpicture}
\draw[white] (-0.1,-0.8) -- (4.1,-0.8) -- (4.1,2.1) -- (-0.1,2.1) -- (-0.1,-0.8); 
\filldraw[fill=white!80!black, draw=white] (4,0) arc (0:180:0.4) -- (4,0);
\filldraw[fill=white!80!black, draw=white] (0,0) arc (180:0:1.6) -- (0,0);
\draw (0,0) arc (180:0:2) -- (0,0);
\foreach \x in {1,...,31}
	\draw (0,0) arc (180:0:0.05*\x);
	
\foreach \x in {1,...,7}
	\draw (4,0) arc (0:180:0.05*\x);

\draw[red, thick] (0,0) arc (180:0:1.6);
\draw[blue, thick] (4,0) arc (0:180:0.4);
\draw (2,0) node[anchor=north] {$t=\frac{1}{5}$};
\end{tikzpicture} \qquad 
\begin{tikzpicture}
\draw[white] (-0.1,-0.8) -- (4.1,-0.8) -- (4.1,2.1) -- (-0.1,2.1) -- (-0.1,-0.8); 
\filldraw[fill=white!80!black, draw=white] (4,0) arc (0:180:0.8) -- (4,0);
\filldraw[fill=white!80!black, draw=white] (0,0) arc (180:0:1.2) -- (0,0);
\draw (0,0) arc (180:0:2) -- (0,0);
\foreach \x in {1,...,23}
	\draw (0,0) arc (180:0:0.05*\x);
	
\foreach \x in {1,...,15}
	\draw (4,0) arc (0:180:0.05*\x);

\draw[red, thick] (0,0) arc (180:0:1.2);
\draw[blue, thick] (4,0) arc (0:180:0.8);
\draw (2,0) node[anchor=north] {$t=\frac{2}{5}$};
\end{tikzpicture}
\begin{tikzpicture}
\draw[white] (-0.1,-0.8) -- (4.1,-0.8) -- (4.1,2.1) -- (-0.1,2.1) -- (-0.1,-0.8); 
\filldraw[fill=white!80!black, draw=white] (4,0) arc (0:180:1.2) -- (4,0);
\filldraw[fill=white!80!black, draw=white] (0,0) arc (180:0:0.8) -- (0,0);
\draw (0,0) arc (180:0:2) -- (0,0);
\foreach \x in {1,...,15}
	\draw (0,0) arc (180:0:0.05*\x);
	
\foreach \x in {1,...,23}
	\draw (4,0) arc (0:180:0.05*\x);

\draw[red, thick] (0,0) arc (180:0:0.8);
\draw[blue, thick] (4,0) arc (0:180:1.2);
\draw (2,0) node[anchor=north] {$t=\frac{3}{5}$};
\end{tikzpicture} \qquad
\begin{tikzpicture}
\draw[white] (-0.1,-0.8) -- (4.1,-0.8) -- (4.1,2.1) -- (-0.1,2.1) -- (-0.1,-0.8); 
\filldraw[fill=white!80!black, draw=white] (4,0) arc (0:180:1.6) -- (4,0);
\filldraw[fill=white!80!black, draw=white] (0,0) arc (180:0:0.4) -- (0,0);
\draw (0,0) arc (180:0:2) -- (0,0);
\foreach \x in {1,...,7}
	\draw (0,0) arc (180:0:0.05*\x);
	
\foreach \x in {1,...,31}
	\draw (4,0) arc (0:180:0.05*\x);

\draw[red, thick] (0,0) arc (180:0:0.4);
\draw[blue, thick] (4,0) arc (0:180:1.6);
\draw (2,0) node[anchor=north] {$t=\frac{4}{5}$};
\end{tikzpicture} \qquad
\begin{tikzpicture}
\draw[white] (-0.1,-0.8) -- (4.1,-0.8) -- (4.1,2.1) -- (-0.1,2.1) -- (-0.1,-0.8); 
\filldraw[fill=white!80!black, draw=black] (0,0) arc (180:0:2) -- (0,0);
\foreach \x in {1,...,39}
	\draw (4,0) arc (0:180:0.05*\x);
\draw (2,0) node[anchor=north] {$t=1$};
\end{tikzpicture}
\caption{Six frames depicting a continuous $\NRclosed$-mutation. (All arcs have orientation left to right.) In the frames between $t=0$ and $t=1$ we mutate the red arc to the blue arc. The first and sixth frames are $T$ and $T'$, respectively. The other four frames are at time $\frac{i}{5}$ for $i\in\{1,2,3,4\}$. We include $\approx 40$ arcs of the uncountably many in the same way one includes level curves in a topographical map.}\label{fig:continuous mutation picture}
\end{figure}

\subsection{Space of Mutations}\label{sec:space of mutations}
In this section we define the space of mutations (Definition \ref{def:space of mutations}) which generalizes the exchange graph of a cluster structure.
The intent is to view mutation paths (Definition \ref{def:mutation path}) as paths in a topological space just as a sequence of mutations of a cluster structure forms a path in the exchange graph.
This majority of this section is for cluster theories in general.
However, its purpose is to study $\mathbf E$-clusters in the future and so we return our attention to $\mathbf E$-clusters at the end of the section.

Since the class of indecomposable objects in $\mathcal C$ form a set and $\C$ is Krull--Schmidt, we see $\mathcal C$ is small.
In particular, the class of morphisms in $\mathcal C$ is a set.
Denote the set of mutations by $(\mathscr T_{\mathbf P}(\mathcal C))_1$.

\begin{definition}\label{def:induced path}
	Let $\overline{\mu}$ be a $\mathbf P$-mutation path and denote by $\pmu$ the induced function from $[0,1]$ to $(\mathscr T_{\mathbf P}(\mathcal C))_1$.
\end{definition}

\begin{definition}\label{def:space of mutations}
We define the set $\mathbf P(\mathcal C)\subset (\mathscr T_{\mathbf P}(\mathcal C))_1$ to be the set containing all (trivial) $\mathbf P$-mutations.

We give the set of $\mathbf P$-mutations a topology in the following way.
Consider $[0,1]$ with the usual topology.
A set $U\subset \mathbf P(\mathcal C)$ is called \ul{open} if, for all $\pmu:[0,1]\to \mathbf P(\mathcal C)$ induced by a $\mathbf P$-mutation $\overline{\mu}$, $\pmu^{-1}(U)$ is open in $[0,1]$.
We call $\mathbf P(\mathcal C)$ the \ul{space of $\mathbf P$-mutations}.
\end{definition}

\begin{proposition}\label{prop:space of mutations}
Then the open sets in Definition \ref{def:space of mutations} form a topology on $\mathbf P(\mathcal C)$.
\end{proposition}
\begin{proof}
Trivially, both $\emptyset$ and $\mathbf P(\mathcal C)$ are open.
Suppose $\pmu:[0,1]\to \mathbf P(\mathcal C)$ is induced by a $\mathbf P$-mutation path $\overline{\mu}$.
Let $\{U_1,\ldots,U_n\}$ be open in $\mathbf P(\mathcal C)$.
Since $\displaystyle\bigcap_{i=1}^n \pmu^{-1}(U_i) = \pmu^{-1}\left( \bigcap_{i=1}^n U_i\right)$, we see that $\bigcap_{i=1}^n U_i$ is open in $\mathbf P(\mathcal C)$.
Now consider a collection $\{U_\alpha\}$ of open sets in $\mathbf P(\mathcal C)$.
Since $\displaystyle\bigcup_\alpha \pmu^{-1}(U_\alpha) = \pmu^{-1}\left( \bigcup_\alpha U_\alpha \right)$, we see that $\bigcup_\alpha U_\alpha$ is open in $\mathbf P(\mathcal C)$.
This concludes the proof.
\end{proof}

\begin{remark}\label{rmk:cluster endpoints}
We consider a $\mathbf P$-cluster $T$ to be the trivial mutation $T\to T$ in $\mathbf P(\mathcal C)$.
We wish to consider paths that start and end at clusters rather than at mutations (see Proposition \ref{prop:path endpoints}).
\end{remark}

\begin{proposition}\label{prop:not Hausdorff}
The space of $\mathbf P$-mutations is not Hausdorff.
\end{proposition}
\begin{proof}
Let $\mu:T\to (T\setminus\{X\})\cup \{Y\}$ be a $\mathbf P$-muation.
Let $\overline{\mu}$ be the $\mathbf P$-mutation path that induces the path $\pmu$ given by
\begin{displaymath} \pmu(t) = \left\{\begin{array}{ll} T & t<1\\ \mu & t=1\end{array}\right.\end{displaymath}
Let $U$ be an open set that contains $\mu$.
If $T\notin U$ then $\pmu^{-1}(U)$ is not open.
This would be a contradiction and so $T\in U$.
Thus, for any $\mathbf P$-mutation $\mu:T\to T'$ and open set $U$ containing $\mu$, $T,T'\in U$ as well.
Therefore, $\mathbf P(\mathcal C)$ is not Hausdorff.
\end{proof}

\begin{proposition}\label{prop:path endpoints}
Let $p:[0,1]\to \mathbf P(\mathcal C)$ be a path in $\mathbf P(\mathcal C)$.
Then there is a path $q:[0,1]\to \mathbf P(\mathcal C)$ whose endpoints are clusters (see Remark \ref{rmk:cluster endpoints}) such that $p$ and $q$ are homotopic.
\end{proposition}
\begin{proof}
Let $p:[0,1]\to \mathbf P(\mathcal C)$ be a path in $\mathbf P(\mathcal C)$, let $T_0$ be the source of $p(0)$, and let $T_1$ the target of $p(1)$.

For any $0<\e<<\frac{1}{2}$, let $q_\e:[0,1]\to\mathbf P(\mathcal C)$ be the path given by:
\begin{displaymath}
q_\e(t) = \left\{\begin{array}{ll}
T_0 & \text{if }t<\e \\
T_1 & \text{if }(1-\e)<t \\
p\left( (t-\frac{1}{2})(1-2\e) + \frac{1}{2}\right) & \text{if }\e\leq t \leq (1-\e)
\end{array}\right.
\end{displaymath}
We see that $q_\e$ is homotopic to the composition of three paths.
The first is constant at $T_0$ except the last point is $p(0)$.
The second is $p$.
The third is constant at $T_1$ except the first point is $p(1)$.
In particular, the first and third path are induced by $\mathbf P$-mutation paths.
Thus, $q_\e$ is indeed a path.
We just say $q_0=p$.

Fix a $0<\e<<\frac{1}{2}$.
Let $H:[0,1]\times [0,1]\to \mathbf P(\mathcal C)$ be given by:
\begin{displaymath}
H(t,s) := q_{s\e}(t).
\end{displaymath}
Let $U$ be open in $\mathbf P(\mathcal C)$.
If the inverse image of $U$ does not contain $p(0)$ or $p(1)$ then $H^{-1}(U)$ is open in $[0,1]\times[0,1]$.

Now suppose $U$ contains $p(0)$.
By the proof of Proposition \ref{prop:not Hausdorff} we see that $T_0\in U$ as well.
Similarly, if $p(1)\in U$ then $T_1\in U$.
Therefore, if $U$ is open in $\mathbf P(\mathcal C)$ then $H^{-1}(U)$ is open in $[0,1]\times[0,1]$, completing the proof.
\end{proof}

\begin{definition}\label{def:reachable}
Let $T_1$ and $T_2$ be $\mathbf P$-clusters of $\mathcal C$.
\begin{enumerate}
\item We say $T_2$ is \ul{reachable} from $T_1$ if there is a path $p:[0,1]\to \mathbf P(\mathcal C)$ such that $p(0)=T_1$ and $p(1)=T_2$.
\item We say $T_2$ is \ul{strongly reachable} from $T_1$ if there is a $\mathbf P$-mutation path $\overline{\mu}$ that (i) comes from a long sequence of continuous $\mathbf P$-mutations and (ii) indues a path $\pmu:[0,1]\to \mathbf P(\mathcal C)$ such that $\pmu(0)=T_1$ and $\pmu(1)=T_2$.
\end{enumerate}
\end{definition}

\begin{theorem}\label{thm:projectives reach injectives}
Let $\AR$ be the continuous quiver of type $A$ with straight descending orientation.
The cluster of injectives, $\mathcal Inj$ is strongly reachable from the cluster of projectives, $\mathcal Proj$.
\end{theorem}
\begin{proof}
In Proposition \ref{prop:proj to inj} we see there is a sequence of $\mathbf E$-mutations $\{\mu_1,\mu_2\}$ to mutate $\mathcal Proj$ to $\mathcal Inj$.
Choose some $0<\e<<\frac{1}{2}$ and note that a sequence of $\mathbf E$-mutations is also a long sequence of $\mathbf E$-mutations.
Then, as in Proposition \ref{prop:sequence of continuous mutation is a mutation path}, we have a $\mathbf E$-mutation path $\overline{\mu}$ with source $\mathcal Proj$ and target $\mathcal Inj$.
\end{proof}

\subsubsection*{Open Questions}
Let $\ARS$ be a continuous quiver of type $A$ and $\mathscr T_{\mathbf E}(\CARS)$ the $\mathbf E$-cluster theory of $\mathcal C$ (Definition \ref{def:E-cluster}).
\begin{itemize}
\item Is the space $\mathbf E(\CARS)$ path connected?
\item If $\mathbf E(\CARS)$ is not path-connected, what do its path components look like? What does the path component containing the cluster of projectives look like?
\item If $\mathbf E(\CARS)$ is not path-connected, are there clusters $T$ and $T'$ that are reachable but not strongly reachable from one another.
\item If $\mathbf E(\CARS)$ is path connected, which clusters are strongly reachable from the cluster of projectives?
\end{itemize}

\end{document}